\documentclass[11pt]{article}

\usepackage[margin=1in]{geometry}
\usepackage{amsmath,amssymb,amsthm}
\usepackage{booktabs}
\usepackage[hidelinks]{hyperref}
\hypersetup{
  pdftitle={Entrywise nonnegativity of the inverse relative gain array in dimension six},
  pdfauthor={Jeffrey Uhlmann and Jie Wang},
  pdfsubject={Inverse relative gain arrays, positivity, and structural proofs},
  pdfkeywords={inverse relative gain array, positive definite matrix, sum of squares, majorization}
}

\newtheorem{theorem}{Theorem}[section]
\newtheorem{lemma}[theorem]{Lemma}
\newtheorem{proposition}[theorem]{Proposition}
\newtheorem{corollary}[theorem]{Corollary}
\theoremstyle{definition}

\newtheorem{example}[theorem]{Example}
\theoremstyle{remark}
\newtheorem{remark}[theorem]{Remark}
\numberwithin{equation}{section}

\DeclareMathOperator{\Diag}{Diag}
\DeclareMathOperator{\diag}{diag}
\DeclareMathOperator{\Sym}{Sym}
\DeclareMathOperator{\tr}{tr}
\DeclareMathOperator{\adj}{adj}
\DeclareMathOperator{\Pf}{Pf}
\DeclareMathOperator{\RGA}{RGA}
\DeclareMathOperator{\IRGA}{IRGA}
\newcommand{\R}{\mathbb{R}}
\newcommand{\Q}{\mathbb{Q}}
\newcommand{\one}{\mathbf{1}}

\title{Entrywise nonnegativity of the inverse relative gain array
in dimension six}
\author{Jeffrey Uhlmann\thanks{Department of Electrical Engineering and Computer
Science, University of Missouri-Columbia. Email: \href{mailto:uhlmannj@missouri.edu}{uhlmannj@missouri.edu}.} \and Jie Wang\thanks{State Key Laboratory of Mathematical Sciences, Academy of Mathematics and Systems Science, Chinese Academy of Sciences, Beijing, China. Email: \href{mailto:wangjie212@amss.ac.cn}{wangjie212@amss.ac.cn}.}}
\date{October 14, 2026}

\begin{document}
\maketitle

\begin{abstract}
We prove that $(G\circ G^{-1})^{-1}$ is entrywise nonnegative for every
real symmetric positive definite matrix $G$ of order at most six. Its
entries are strictly positive exactly within the connected components of
the nonzero-entry graph of $G$. An explicit rational example of order
seven establishes sharpness. We give two complementary structural
proofs. The analytic proof combines a trace obstruction with continuation
of positive column forms and a bordering identity. The algebraic proof
establishes a stronger four-dimensional polarized inequality through an
explicit ten-square polynomial certificate. A common framework relates
biquadratic positivity, column forms, positive semidefinite precision
loads, and inverse relative gain array nonnegativity in consecutive
dimensions. Further consequences include exact load-deficit identities,
polynomial sum-of-squares certificates with the covariance fixed, and a
reverse Schur--Horn majorization inequality with the same sharp dimension
bound. In every dimension, sufficiently large powers of the inverse
relative gain array are doubly stochastic.
\end{abstract}



\section{Introduction}
\label{sec:introduction}

The inverse relative gain array conjecture asks whether
$(G\circ G^{-1})^{-1}$ is entrywise nonnegative for every real symmetric
positive definite matrix $G$ of order at most six. We prove the conjecture
in its full range, determine the exact support of the inverse, and give
two structural proofs. The proofs also yield positivity and equality
statements for precision loads and a sharp reverse majorization theorem.

In our earlier work \cite{UhlmannWang}, sparse sum-of-squares (SOS)
methods proved nonnegativity through order four, and an order-seven
counterexample established the upper limit of the conjectured range.
The order-six formulation considered there involved a polynomial in
15 variables, of degree 40, with 676,505 monomials. We did not obtain
even a numerical SOS decomposition of that polynomial and proposed it
as a challenge problem. Remark 9.2 of \cite{UhlmannWang} anticipated exploiting
special structure either to avoid SOS methods or to reduce the complexity
of the polynomials requiring certification. The present proofs realize
both possibilities: one is analytic, and the other uses an exact SOS
identity after reduction to a four-dimensional inequality.

The reduction has two distinct consequences. The analytic argument gives
the full support theorem, while the algebraic certificate establishes a
stronger polarized inequality. Neither proof requires a sum-of-squares
certificate for the original degree-40 polynomial.
This provides a concrete setting in which to examine AI-assisted
structural discovery. Interactions with ChatGPT supported the development of arguments organized around column forms, precision loads, and bordering. The resulting proofs resolve the conjecture without direct certification of the earlier order-six polynomial.

\subsection{The mathematical result}

For a nonsingular real matrix $G$, the relative gain array is
\[
\RGA(G)=G\circ G^{-\mathrm{T} },
\]
where $\circ$ denotes the Hadamard product. It was introduced by Bristol
\cite{Bristol} to measure interaction in multivariable process control;
Johnson and Shapiro \cite{JohnsonShapiro} studied its matrix-theoretic
properties. For a real symmetric positive definite matrix $G$, write
\[
K_G=G\circ G^{-1},\qquad R_G=K_G^{-1}=\IRGA(G).
\]
The Schur product theorem gives $K_G,R_G\succ0$ in every order.
Both matrices have unit row and column sums. Entrywise nonnegativity
is the remaining condition for $R_G$ to be doubly stochastic.

The covariance graph of $G$ has vertices $1,\ldots,n$ and an edge between
distinct indices $i,j$ when $G_{ij}\ne0$. Its connected components are
called covariance components. We call $G$ irreducible when this graph is
connected. The word covariance imposes no probabilistic assumption.

\begin{theorem}[Positivity, support, and sharpness]
\label{thm:main}
Let $G$ be a real symmetric positive definite matrix of order $1\le n\le6$.
Then $R_G$ is symmetric positive definite, entrywise nonnegative, and
doubly stochastic. More precisely,
\begin{equation}
(R_G)_{ij}>0
\quad\Longleftrightarrow\quad
 i,j\text{ belong to the same covariance component of }G.
\label{eq:support}
\end{equation}
Entries between distinct components are zero. The dimension bound is
sharp: for every $n\ge7$, some real symmetric positive definite matrix
of order $n$ has a negative IRGA entry.
\end{theorem}

The theorem resolves the real symmetric positive definite case of the
conjecture in \cite{UhlmannWang} and determines every equality case for
individual entries in the proved dimension range. It imposes no
restrictions on the signs of the entries of $G$.

\subsection{Two structural mechanisms}

For a positive definite matrix $C$ and a coordinate $i$, define the
symmetric column form
\[
S_{C,i}=\Sym\bigl(C\Diag(R_Ce_i)\bigr),\qquad
u^{\mathrm{T} }S_{C,i}u=[R_C(u\circ Cu)]_i,
\]
where $\Sym(M)=(M+M^{\mathrm{T} })/2$ and $\Diag(w)$ is the diagonal matrix with
diagonal $w$. Positive column forms through order five provide the common
link to IRGA nonnegativity through order six.

The analytic proof establishes strict positivity of these forms directly
on irreducible inputs. The key obstruction to losing positive
definiteness is
\[
\tr(M^{-1}M^{\mathrm{T} })\le m-4
\]
when $M$ is invertible and $\Sym(M)$ is positive semidefinite and
singular. At a hypothetical boundary along an irreducibility-preserving
path, graph connectivity forces the selected IRGA column to be strictly
positive. For the matrix constructed from that column, the same trace
equals the reciprocal of a diagonal IRGA entry and is therefore greater
than one. This is incompatible with the bound when $m\le5$.
A bordering identity then proves positivity through order six and gives
the exact support statement~\eqref{eq:support}.

The algebraic proof establishes the stronger polarized inequality
\begin{equation}
 d_C(u)^{\mathrm{T} }R_Cd_C(v)
 \ge h_C(u,v)^{\mathrm{T} }R_Ch_C(u,v)
 \ge \frac14(u^{\mathrm{T} }Cv)^2,
 \qquad C\in S^4_{++},
\label{eq:intro-polarized}
\end{equation}
where $S^m_{++}$ denotes the real symmetric positive definite matrices
of order $m$, and
\[
d_C(u)=u\circ(Cu),\qquad
h_C(u,v)=\frac12\bigl(u\circ(Cv)+v\circ(Cu)\bigr).
\]
The polarization defect is a quadratic form on decomposable bivectors.
Adding a multiple of the Pl\"ucker relation, normalizing a Cholesky
factor, and applying an invertible congruence produce a ten-square
identity with positive rational weights. Reversible bordering links the
weaker bound in~\eqref{eq:intro-polarized} to column-form positivity in
order five and IRGA nonnegativity in order six.

The two proofs share the elementary load and bordering framework.
The analytic argument requires no polynomial certificate. The stronger
polarized inequality and its explicit certificate supply additional
information beyond the nonnegativity theorem.

\subsection{Consequences and organization}

The load framework gives exact column-sum deficits under arbitrary
positive semidefinite precision loads, without assuming that the loaded
inverse is entrywise nonnegative. With an irreducible covariance of
order at most five fixed, each deficit numerator is a polynomial SOS
in a load factor of arbitrary width. Its coefficients depend on the
fixed covariance; this result does not supply a joint-parameter SOS
certificate for the earlier order-six numerator.

We also establish the sharp reverse Schur--Horn inequality
\[
\lambda\prec\diag\bigl(P\Diag(\lambda)P^{-1}\bigr),
\qquad P\in S^n_{++},\quad\lambda\in\R^n,\quad n\le6,
\]
extending the reverse majorization results of
\cite{Uhlmann2023,UhlmannPartial} to the full conjectured range.
For a fixed positive definite $P$, this assertion for every real spectrum
is equivalent to $R_P\ge0$. In every dimension, a commutator identity
identifies the fixed space of $R_G$ and the limit of its powers.
Those powers are eventually positive within every covariance component.

Section~\ref{sec:framework} develops the column, load, and bordering
identities needed for the analytic proof. Section~\ref{sec:analytic}
proves Theorem~\ref{thm:main}, including sharpness.
Section~\ref{sec:algebraic} gives the consecutive-dimension equivalence,
the polarized inequality, and the second proof of nonnegativity.
Sections~\ref{sec:loads} and~\ref{sec:applications} treat precision loads
and the averaging and majorization consequences.
Section~\ref{sec:discussion} summarizes the mathematical implications
and further questions. The appendices contain the polynomial certificate and the
Pfaffian construction for the fixed-covariance SOS result.

\section{Column forms, precision loads, and bordering}
\label{sec:framework}

Write $S^m$, $S^m_+$, and $S^m_{++}$ for the real symmetric,
positive semidefinite, and positive definite matrices of order $m$.
The symbols $\succeq$ and $\succ$ denote the Loewner order; $\ge$ between
vectors or matrices denotes entrywise comparison. All vectors are
columns. We use $\one$ and $e_i$ for the all-ones vector and the $i$th
coordinate vector. The notation $\Diag(x)$ denotes the diagonal matrix
with diagonal $x$, whereas $\diag(M)$ denotes the diagonal of a square
matrix as a vector. Set $\Sym(M)=(M+M^{\mathrm{T} })/2$.
The classical adjugate $\adj(M)$ is the transpose of the cofactor matrix;
it satisfies $M\adj(M)=(\det M)I$ for every square matrix $M$.

For $C\in S^m_{++}$, define
\begin{equation}
K_C=C\circ C^{-1},\qquad R_C=K_C^{-1},\qquad
 d_C(u)=\Diag(u)Cu,
\label{eq:notation}
\end{equation}
and use $h_C$ as defined in the introduction. The Frobenius and spectral
matrix norms are denoted by $\|\cdot\|_F$ and $\|\cdot\|_2$.

\subsection{Elementary identities}

If $C\succ0$ and $E\succeq0$, a rank-one decomposition
$E=\sum_\ell v_\ell v_\ell^{\mathrm{T} }$ gives
\begin{equation}
 C\circ E=\sum_\ell\Diag(v_\ell)C\Diag(v_\ell)\succeq0.
\label{eq:schur}
\end{equation}
If $E\succ0$, choose $\eta>0$ with $E\succeq\eta I$ to obtain
$C\circ E\succeq\eta\Diag(C_{11},\ldots,C_{mm})\succ0$.
Also, $\tr(QE)\ge0$ for $Q,E\succeq0$, and when $Q\succ0$ equality
holds exactly when $E=0$.

Consequently $K_C,R_C\succ0$. Symmetry of $C^{-1}$ gives
\begin{equation}
 (K_C\one)_i=\sum_j C_{ij}(C^{-1})_{ji}=1,
 \qquad K_C\one=R_C\one=\one.
\label{eq:unitsums}
\end{equation}
These identities hold in every dimension and do not imply inverse
nonnegativity.

For a nonsingular diagonal $D$ and a permutation matrix $U$,
\begin{align}
 K_{DCD}&=K_C,&R_{DCD}&=R_C,&d_{DCD}(u)&=d_C(Du),
 \label{eq:diagonal-invariance}\\
 R_{UCU^{\mathrm{T} }}&=UR_CU^{\mathrm{T} },&d_{UCU^{\mathrm{T} }}(u)&=Ud_C(U^{\mathrm{T} }u).
 \label{eq:permutation-invariance}
\end{align}
Both $K$ and $R$ preserve direct sums. Thus componentwise positivity on
irreducible matrices determines the zero pattern for reducible inputs.

\subsection{Column forms}

For each coordinate $i$, define
\begin{equation}
 S_{C,i}=\Sym\bigl(C\Diag(R_Ce_i)\bigr).
\label{eq:column-forms}
\end{equation}
For $Y\in S^m$ and $u\in\R^m$, expansion gives
\begin{align}
 [R_Cd_C(u)]_i&=u^{\mathrm{T} }S_{C,i}u,
 \label{eq:dictionary}\\
 [R_C\diag(CY)]_i&=\tr(S_{C,i}Y).
 \label{eq:trace-dictionary}
\end{align}
Thus the signs of the vector $R_C\diag(CY)$ on positive semidefinite
inputs $Y$ are determined by the column forms.

\subsection{A representation for arbitrary positive semidefinite loads}

A positive semidefinite precision load replaces $C^{-1}$ by
$C^{-1}+E$, with $E\succeq0$, in the product
\[
 C\circ(C^{-1}+E)=K_C+C\circ E.
\]
The covariance factor $C$ is held fixed. The next lemma expresses the
resulting deficit in unit sums through the unloaded column forms.
It is valid in every dimension, without an inverse sign assumption.

\begin{lemma}[Positive semidefinite load representation]
\label{lem:load-representation}
Let $C\in S^m_{++}$ and $E\succeq0$. Put
\begin{equation}
 B=K_C+C\circ E,\qquad q=B^{-1}\one,\qquad t=\one-q,
\label{eq:load-defs}
\end{equation}
and define
\begin{equation}
 X=\Diag(q)E\Diag(q)
   +\Diag(t)C^{-1}\Diag(t).
\label{eq:load-X}
\end{equation}
Then $B\succ0$, $X\succeq0$, and
\begin{equation}
 \diag(CX)=K_Ct,\qquad
 t=R_C\diag(CX),\qquad t_i=\tr(S_{C,i}X).
\label{eq:load-representation}
\end{equation}
Moreover, $X=0$ if and only if $E=0$.
\end{lemma}

\begin{proof}
Equation~\eqref{eq:schur} gives $B\succ0$, and the two terms in
\eqref{eq:load-X} are positive semidefinite congruences regardless of
the signs of $q$ and $t$. From $Bq=\one=K_C\one$,
\[
 K_Ct=(C\circ E)q.
\]
The diagonals contributed by the two terms in $X$ are respectively
$q\circ(K_Ct)$ and $t\circ(K_Ct)$. Since $q+t=\one$, they sum to
$K_Ct$. This proves the first identity in
\eqref{eq:load-representation}; the others follow by multiplication
by $R_C$ and \eqref{eq:trace-dictionary}.

If $X=0$, its second positive semidefinite summand vanishes.
Its diagonal entries are $t_i^2(C^{-1})_{ii}$, so $t=0$.
Then $q=\one$ and $X=E=0$. Conversely, $E=0$ gives
$q=R_C\one=\one$, $t=0$, and $X=0$.
\end{proof}

\begin{proposition}[Column forms and load bounds]
\label{prop:load-equivalence}
For a fixed $C\in S^m_{++}$, the following conditions are equivalent:
\begin{enumerate}
\item[(i)] $S_{C,i}\succeq0$ for every $i$;
\item[(ii)] $(K_C+C\circ E)^{-1}\one\le\one$ for every $E\succeq0$.
\end{enumerate}
If $S_{C,i}\succ0$, then the $i$th inequality in (ii) is strict
for every nonzero $E\succeq0$.
\end{proposition}

\begin{proof}
Lemma~\ref{lem:load-representation} gives (i)$\Rightarrow$(ii), and
positive definiteness of $S_{C,i}$ gives strictness since $X\ne0$
when $E\ne0$. Conversely, take $E=\varepsilon uu^{\mathrm{T} }$ with
$\varepsilon>0$. Expanding the inverse at $K_C$ gives
\[
 \one-(K_C+\varepsilon\Diag(u)C\Diag(u))^{-1}\one
   =\varepsilon R_Cd_C(u)+O(\varepsilon^2).
\]
Condition (ii), division by $\varepsilon$, and passage to the limit
imply $R_Cd_C(u)\ge0$ for every $u$. Equation~\eqref{eq:dictionary}
then proves (i).
\end{proof}

\begin{corollary}[Componentwise strictness]
\label{cor:component-load}
Suppose that for every irreducible component $C_I$ of $C$, all the
forms $S_{C_I,i}$ are positive definite. For $B$ in
\eqref{eq:load-defs} and $i\in I$,
\begin{equation}
 1-(B^{-1}\one)_i\ge0,
 \qquad
 1-(B^{-1}\one)_i=0\quad\Longleftrightarrow\quad E_{I,I}=0.
\label{eq:component-load}
\end{equation}
\end{corollary}

\begin{proof}
Permute $C$ into its component blocks. The matrices $C^{-1}$ and
$C\circ E$ have the same block decomposition, even when $E$ has
nonzero off-block entries. Apply Proposition~\ref{prop:load-equivalence}
to each block $K_{C_I}+C_I\circ E_{I,I}$.
\end{proof}

\subsection{Normalized bordering}

For $C\in S^m_{++}$ and $x\in\R^m$, set
\begin{equation}
 \mathcal{B}(C,x)=
 \begin{pmatrix}1+x^{\mathrm{T} }Cx&x^{\mathrm{T} }C\\Cx&C\end{pmatrix}.
\label{eq:normalized-border}
\end{equation}
Its Schur complement over $C$ is one. Conversely, if
\[
 A=\begin{pmatrix}\alpha&c^{\mathrm{T} }\\c&C\end{pmatrix}\succ0,
 \qquad s=\alpha-c^{\mathrm{T} }C^{-1}c>0,
\]
then with $D=\Diag(s^{-1/2},1,\ldots,1)$ and
$x=s^{-1/2}C^{-1}c$ we have $DAD=\mathcal{B}(C,x)$.
Thus every distinguished coordinate permits this normalization.

\begin{lemma}[Border row]
\label{lem:border}
Let $A=\mathcal{B}(C,x)$, and put
\[
 d=d_C(x),\qquad N=K_C+\Diag(x)C\Diag(x),\qquad
 q=N^{-1}\one,\qquad t=\one-q,\qquad z=\Diag(x)q.
\]
Then
\begin{align}
 X&=zz^{\mathrm{T} }+\Diag(t)C^{-1}\Diag(t)\succeq0,
 &t&=R_C\diag(CX),\label{eq:border-X}\\
 e_1^{\mathrm{T} }R_A&=\sigma^{-1}(1,t^{\mathrm{T} }),
 &\sigma&=1+x^{\mathrm{T} }Cx-d^{\mathrm{T} }N^{-1}d=1+\one^{\mathrm{T} }t>0.
\label{eq:border-row}
\end{align}
\end{lemma}

\begin{proof}
Block inversion gives
\[
 A^{-1}=\begin{pmatrix}1&-x^{\mathrm{T} }\\-x&C^{-1}+xx^{\mathrm{T} }\end{pmatrix},
 \qquad
 K_A=\begin{pmatrix}1+x^{\mathrm{T} }Cx&-d^{\mathrm{T} }\\-d&N\end{pmatrix}.
\]
The Schur complement $\sigma$ is positive because $K_A\succ0$.
Since $N\one=\one+d$, we have $t=N^{-1}d$. A second block
inversion gives the row in \eqref{eq:border-row}. Moreover,
$d^{\mathrm{T} }t=\one^{\mathrm{T} }d-\one^{\mathrm{T} }t=x^{\mathrm{T} }Cx-\one^{\mathrm{T} }t$, which proves the second
expression for $\sigma$. Equation~\eqref{eq:border-X} is
Lemma~\ref{lem:load-representation} with $E=xx^{\mathrm{T} }$.
\end{proof}

\section{An analytic proof of positivity, support, and sharpness}
\label{sec:analytic}

We first prove that the column forms are positive definite on
irreducible covariances through order five. A continuation argument
excludes a singular boundary by comparing a trace identity with a
spectral bound.
Normalized bordering then gives IRGA positivity and exact support
through order six, and an explicit order-seven example proves
sharpness. The proof uses the load and bordering framework of
Section~\ref{sec:framework}.

\subsection{A trace obstruction at the boundary}

\begin{lemma}[Singular symmetric part]
\label{lem:trace-bound}
Let $M\in\R^{m\times m}$ be invertible and suppose that
$\Sym(M)\succeq0$ is singular. Then
\begin{equation}
 \tr(M^{-1}M^{\mathrm{T} })\le m-4.
\label{eq:trace-bound}
\end{equation}
\end{lemma}

\begin{proof}
The hypotheses are impossible for $m=1$. Set $U=M^{-1}M^{\mathrm{T} }$ and
$S=\Sym(M)$. We first show that every eigenvalue of $U$ has modulus
one. If $Uz=\lambda z$ for $0\ne z\in\mathbb{C}^m$, then
$M^{\mathrm{T} }z=\lambda Mz$. Put $a=z^*Mz$. Since $M$ is real,
\[
 \overline a=z^*M^{\mathrm{T} }z=\lambda a.
\]
If $a\ne0$, this gives $|\lambda|=1$. If $a=0$, then
$z^*Sz=\operatorname{Re}(a)=0$. Since $S\succeq0$, this implies
$Sz=0$, also for complex $z$. Hence $M^{\mathrm{T} }z=-Mz$, and invertibility
of $M$ gives $\lambda=-1$.

Next, $U+I=2M^{-1}S$ is singular, so $-1$ is an eigenvalue.
Let its algebraic multiplicity be $r\ge1$. The nonreal eigenvalues
of the real matrix $U$ occur in conjugate pairs, each with product
one. The only possible real eigenvalues are $1$ and $-1$. Therefore
\[
 1=\det U=\frac{\det M^{\mathrm{T} }}{\det M}=(-1)^r.
\]
Thus $r$ is even and $r\ge2$. The trace is the sum of eigenvalues
with algebraic multiplicity, whether or not $U$ is diagonalizable.
Taking real parts gives
\[
 \tr U\le-r+(m-r)=m-2r\le m-4.
\]
\end{proof}

\begin{remark}
Algebraic multiplicity is essential. For $m\ge2$, the matrix
\[
 M=\begin{pmatrix}0&1\\-1&1\end{pmatrix}\oplus I_{m-2},
 \qquad
 M^{-1}M^{\mathrm{T} }=\begin{pmatrix}-1&-2\\0&-1\end{pmatrix}\oplus I_{m-2}
\]
attains equality in \eqref{eq:trace-bound}, although the $-1$
eigenspace of $M^{-1}M^{\mathrm{T} }$ has dimension one. No normality assumption
is needed.
\end{remark}

\subsection{Trace identities for an unloaded column}

\begin{lemma}[Column moments]
\label{lem:column-moments}
Let $C=FF^{\mathrm{T} }\succ0$, with $F$ real and invertible. Suppose that
$w=R_Ce_i$ is entrywise strictly positive, and put
\[
 D=\Diag(w),\qquad M=F^{\mathrm{T} }DF^{-\mathrm{T} },\qquad S=\Sym(M).
\]
Then $M$ is invertible and
\begin{align}
 \tr M&=1,&\|M\|_F^2&=w_i,
 \label{eq:column-moments}\\
 \tr(M^{-1}M^{\mathrm{T} })&=\frac1{w_i},&FSF^{\mathrm{T} }&=\Sym(CD).
 \label{eq:column-trace}
\end{align}
\end{lemma}

\begin{proof}
The matrix $M$ is similar to the positive diagonal matrix $D$.
Equation~\eqref{eq:unitsums} gives $\one^{\mathrm{T} }w=1$, so $\tr M=1$.
Using $C^{-1}=F^{-\mathrm{T} }F^{-1}$ and cyclicity of the trace,
\[
 \|M\|_F^2=\tr(DC^{-1}DC)=w^{\mathrm{T} }K_Cw=w_i.
\]
Likewise,
\[
 \tr(M^{-1}M^{\mathrm{T} })
 =\tr(D^{-1}C^{-1}DC)
 =\sum_j\frac{(K_Cw)_j}{w_j}
 =\frac1{w_i}.
\]
All divisions are justified by $w>0$. Finally,
$FMF^{\mathrm{T} }=CD$ and $FM^{\mathrm{T} }F^{\mathrm{T} }=DC$, proving the congruence.
\end{proof}

Replacing $F$ by $FO$ for an orthogonal matrix $O$ replaces $M$ by
$O^{\mathrm{T} }MO$. In particular, the positivity of its symmetric part is
independent of the chosen square factor.

\subsection{Positive column forms through order five}

The following observation supplies strict column positivity at a
hypothetical semidefinite boundary.

\begin{lemma}[Zero propagation]
\label{lem:column-zero-propagation}
Let $C\in S^m_{++}$ be irreducible and let $w\in\R^m$ satisfy
$\one^{\mathrm{T} }w=1$. If $\Sym(C\Diag(w))\succeq0$, then $w>0$.
\end{lemma}

\begin{proof}
Put $T=\Sym(C\Diag(w))$. Its diagonal entries are
$T_{jj}=C_{jj}w_j$, so $w\ge0$. If $w_j=0$, then $T_{jj}=0$.
A zero diagonal entry of a positive semidefinite matrix forces the
corresponding row and column to vanish. Consequently,
\[
 0=T_{jk}=\frac12C_{jk}(w_j+w_k)=\frac12C_{jk}w_k
 \qquad\text{for every }k.
\]
Thus every neighbor of $j$ in the covariance graph also has zero
weight. Connectedness propagates this to every vertex, contradicting
$\one^{\mathrm{T} }w=1$.
\end{proof}

\begin{proposition}[Positive column forms]
\label{prop:positive-column-forms}
\label{prop:P-to-H}
For every irreducible $C\in S^m_{++}$ with $1\le m\le5$ and every
coordinate $i$,
\[
 S_{C,i}\succ0\qquad\text{and}\qquad R_Ce_i>0.
\]
\end{proposition}

\begin{proof}
The scalar case is immediate. Suppose $2\le m\le5$ and fix an
irreducible $C\in S^m_{++}$ and a coordinate $i$.

Choose a spanning tree of the covariance graph of $C$. Assign signs
$u_j\in\{-1,1\}$ along the tree, starting from a root, so that
$u_ju_k$ has the sign of $C_{jk}$ on every tree edge. Since a tree
has no cycles, these signs can always be assigned. For $\alpha>0$,
put
\begin{equation}
 C_0=I+\alpha uu^{\mathrm{T} },\qquad C_t=(1-t)C_0+tC,
 \qquad 0\le t\le1.
\label{eq:reference-path}
\end{equation}
Every $C_t$ is positive definite. On a selected tree edge,
\[
 (C_t)_{jk}=u_ju_k\bigl((1-t)\alpha+t|C_{jk}|\bigr)\ne0,
\]
so every $C_t$ is irreducible.

Write $b=\alpha^2/(1+m\alpha)$. Direct multiplication gives
\begin{equation}
 C_0^{-1}=I-\frac{\alpha}{1+m\alpha}uu^{\mathrm{T} },\qquad
 K_{C_0}=(1+mb)I-b\one\one^{\mathrm{T} },\qquad
 R_{C_0}=\frac{I+b\one\one^{\mathrm{T} }}{1+mb}.
\label{eq:reference-formulas}
\end{equation}
In particular, $w_0=R_{C_0}e_i>0$. Choose a square factor
$C_0=F_0F_0^{\mathrm{T} }$ and put $M_0=F_0^{\mathrm{T} }\Diag(w_0)F_0^{-\mathrm{T} }$.
Lemma~\ref{lem:column-moments} gives
\begin{equation}
 \left\|M_0-\frac1m I\right\|_F^2
  =(w_0)_i-\frac1m
  =\frac{m-1}{m(1+mb)}.
\label{eq:reference-norm}
\end{equation}
Choose $\alpha=m^2$. Then
\[
 1+mb-m(m-1)=m+1-\frac{m^2}{1+m^3}>0,
\]
so the last quantity in \eqref{eq:reference-norm} is less than
$1/m^2$. Since $\|\Sym(N)\|_2\le\|N\|_F$,
\[
 \left\|\Sym(M_0)-\frac1mI\right\|_2<\frac1m.
\]
Hence $\Sym(M_0)\succ0$, and the congruence in
\eqref{eq:column-trace} gives $S_{C_0,i}\succ0$.

Now set
\[
 w_t=R_{C_t}e_i,\qquad
 S_t=S_{C_t,i}=\Sym(C_t\Diag(w_t)).
\]
The symmetric matrices $S_t$ depend continuously on $t$, since
matrix inversion is continuous on the positive definite cone.
Suppose for contradiction that $S_1$ is not positive definite, and
set
\[
 t_* = \inf\{t\in[0,1]:\lambda_{\min}(S_t)\le0\}.
\]
The set is nonempty and closed because the smallest eigenvalue is
continuous. Since $S_0\succ0$, continuity also gives $t_*>0$.
For every $t<t_*$ we have $S_t\succ0$; hence, by continuity,
$\lambda_{\min}(S_{t_*})=0$ and all eigenvalues of $S_{t_*}$
are nonnegative. Thus $S_{t_*}\succeq0$ is singular, including
when $t_*=1$. Equation~\eqref{eq:unitsums} gives
$\one^{\mathrm{T} }w_{t_*}=1$. Since $C_{t_*}$ is irreducible,
Lemma~\ref{lem:column-zero-propagation} implies $w_{t_*}>0$.
As $m\ge2$, it follows that
\begin{equation}
 0<(w_{t_*})_i<1.
\label{eq:strict-column-diagonal}
\end{equation}

Choose a square factor $C_{t_*}=F_*F_*^{\mathrm{T} }$ and set
$M_*=F_*^{\mathrm{T} }\Diag(w_{t_*})F_*^{-\mathrm{T} }$. This matrix is invertible,
and \eqref{eq:column-trace} shows that $\Sym(M_*)$ is positive
semidefinite and singular. Lemmas~\ref{lem:trace-bound} and
\ref{lem:column-moments} therefore give
\[
 1<\frac1{(w_{t_*})_i}
   =\tr(M_*^{-1}M_*^{\mathrm{T} })
   \le m-4\le1,
\]
a contradiction. Thus $S_{C,i}=S_1\succ0$. Finally,
\[
 C_{jj}(R_Ce_i)_j=(S_{C,i})_{jj}>0
 \qquad\text{for every }j,
\]
which proves $R_Ce_i>0$.
\end{proof}

\begin{remark}
The restriction $m\le5$ enters only through $m-4\le1$ in the
final contradiction. The continuation stays within positive definite
covariances; its hypothetical singular boundary concerns a column form.
\end{remark}

\begin{corollary}[Auxiliary positivity through order five]
\label{cor:auxiliary}
For every irreducible $C\in S^m_{++}$ with $1\le m\le5$ and every
$i$, $R_Ce_i>0$ and $S_{C,i}\succ0$. For arbitrary
$C\in S^m_{++}$ with $1\le m\le5$, $S_{C,i}\succeq0$ and
$R_Cd_C(u)\ge0$ for all $u$.
\end{corollary}

\begin{proof}
The irreducible statement is Proposition~\ref{prop:positive-column-forms}.
For reducible $C$, both $K_C$ and $R_C$ preserve its component
blocks. The selected column vanishes outside its component, and the
corresponding form is positive definite on that block and zero on
the other blocks. Equation~\eqref{eq:dictionary} gives the last
assertion.
\end{proof}

\subsection{Bordering and exact support}

\begin{proposition}[Extension by one coordinate]
\label{prop:extension}
Let $m\ge1$. Suppose that $S_{C,i}\succ0$ for every irreducible
positive definite $C$ of order at most $m$ and every coordinate $i$.
Then the IRGA of every positive definite matrix of order at most
$m+1$ is positive within its covariance components and zero between
distinct components.
\end{proposition}

\begin{proof}
The scalar case is immediate. Let $G\in S^n_{++}$ be irreducible,
where $2\le n\le m+1$, and distinguish any coordinate. By diagonal
and permutation invariance, we may put that coordinate first and
normalize $G$ as $\mathcal B(C,x)$ in
\eqref{eq:normalized-border}. The principal covariance $C$ may be
reducible. Put
\[
 E=xx^{\mathrm{T} },\qquad N=K_C+C\circ E,\qquad t=\one-N^{-1}\one.
\]
For any covariance component $I$ of $C$, irreducibility of $G$
forces $(Cx)_I\ne0$; otherwise that component would have no edge
to its complement in $G$. Since $(Cx)_I=C_Ix_I$, we have
$x_I\ne0$ and hence $E_{I,I}=x_Ix_I^{\mathrm{T} }\ne0$. The hypothesis and
Corollary~\ref{cor:component-load} therefore give $t_i>0$ for every
$i$. Lemma~\ref{lem:border} now yields, with $\sigma>0$,
\[
 e_1^{\mathrm{T} }R_G=\sigma^{-1}(1,t^{\mathrm{T} })>0.
\]
The distinguished coordinate was arbitrary, so $R_G$ is entrywise
strictly positive. For reducible $G$, apply this conclusion to its
irreducible blocks and use preservation of direct sums.
\end{proof}

\subsection{Sharpness of IRGA nonnegativity}
\label{sec:sharpness}

\begin{example}[An order-seven counterexample]
\label{ex:seven}
Let $\one\in\R^5$ and set
\begin{equation}
 G_7=\begin{pmatrix}
 21&3&4\one^{\mathrm{T} }\\
 3&21&4\one^{\mathrm{T} }\\
 4\one&4\one&16I_5
 \end{pmatrix}.
\label{eq:seven-covariance}
\end{equation}
Its Schur complement over $16I_5$ is
$\begin{pmatrix}16&-2\\-2&16\end{pmatrix}\succ0$, so $G_7\succ0$.
Block inversion and the Hadamard product give
\begin{equation}
 K_{G_7}=\begin{pmatrix}
 4/3&1/42&-\one^{\mathrm{T} }/14\\
 1/42&4/3&-\one^{\mathrm{T} }/14\\
 -\one/14&-\one/14&(8/7)I_5
 \end{pmatrix}.
\label{eq:seven-RGA}
\end{equation}
The Schur complement of its last five coordinates is
\[
 H=\frac1{672}\begin{pmatrix}881&1\\1&881\end{pmatrix}.
\]
The upper-left block of $R_{G_7}$ is $H^{-1}$, and therefore
\begin{equation}
 (R_{G_7})_{12}=-\frac1{1155}<0.
\label{eq:seven-negative}
\end{equation}
For every $n>7$, the direct sum $G_7\oplus I_{n-7}$ preserves this
negative entry.
\end{example}

\begin{proof}[Proof of Theorem~\ref{thm:main}]
Proposition~\ref{prop:positive-column-forms} supplies the hypotheses
of Proposition~\ref{prop:extension} with $m=5$. Hence $R_G$ is
entrywise nonnegative and has exactly the support in
\eqref{eq:support} for every $1\le n\le6$. Positive definiteness
follows from the Schur product theorem, and the unit row and column
sums in \eqref{eq:unitsums} make $R_G$ doubly stochastic.
Example~\ref{ex:seven} proves sharpness in every order $n\ge7$.
\end{proof}

Strict positivity is pointwise. There is no uniform positive lower
bound for off-diagonal IRGA entries over all irreducible inputs of
a fixed order $2\le n\le6$, since such entries can tend to zero at a
reducible limit. Example~\ref{ex:seven} will also give the sharpness
of the reverse majorization result in Section~\ref{sec:applications}.

\subsection{Failure of the column-form property in order six}

\begin{example}
\label{ex:auxiliary-six}
Let
\[
 F=\begin{pmatrix}1&0\\\one/10&I_5\end{pmatrix},\qquad C=FF^{\mathrm{T} },
 \qquad\one\in\R^5.
\]
The covariance is irreducible, and direct calculation gives
\[
 K_C=\begin{pmatrix}
 21/20&-\one^{\mathrm{T} }/100\\-\one/100&(101/100)I_5
 \end{pmatrix},\qquad
 w=R_Ce_1=\frac1{106}(101,1,1,1,1,1)^{\mathrm{T} }.
\]
With $D=\Diag(w)$,
\[
 M=F^{\mathrm{T} }DF^{-\mathrm{T} }=\frac1{106}
 \begin{pmatrix}101&-10\one^{\mathrm{T} }\\0&I_5\end{pmatrix}.
\]
For $z=(1,5,5,5,5,5)^{\mathrm{T} }$,
\[
 z^{\mathrm{T} }\Sym(M)z=z^{\mathrm{T} }Mz=\frac{101-250+125}{106}
                         =-\frac{12}{53}<0.
\]
Since $e_1^{\mathrm{T} }\Sym(M)e_1=101/106>0$, the matrix $\Sym(M)$ is
indefinite. The congruence
\[
 F\Sym(M)F^{\mathrm{T} }=\Sym(CD)=S_{C,1}
\]
shows that the column form is indefinite as well, although $R_C$
is entrywise strictly positive by Theorem~\ref{thm:main}.
\end{example}

Thus universal positive semidefiniteness of the column forms ends
at order five, while universal IRGA nonnegativity ends at order six.

\section{A polarized inequality and an algebraic proof}
\label{sec:algebraic}

We establish a polarized inequality that strengthens the quadratic
inequality obtained from the analytic proof. Its explicit certificate
also gives a second proof of IRGA nonnegativity through order six,
using the framework of Section~\ref{sec:framework} and the polynomial
identity in Appendix~\ref{app:certificate}.

\subsection{Positive maps and a centering identity}

For $C\in S^m_{++}$, define the linear map on $S^m$ by
\begin{equation}
 \Phi_C(Y)=\Sym\bigl(\Diag(R_C\diag(CY))C\bigr).
\label{eq:positive-map-definition}
\end{equation}
For $u,v\in\R^m$, expansion gives
\begin{equation}
 v^{\mathrm{T} }\Phi_C(uu^{\mathrm{T} })v=d_C(v)^{\mathrm{T} }R_Cd_C(u).
\label{eq:phi-rank-one}
\end{equation}
Consequently, for a fixed real $\eta$,
\begin{equation}
\begin{aligned}
 &d_C(u)^{\mathrm{T} }R_Cd_C(v)\ge\eta(u^{\mathrm{T} }Cv)^2
          \quad\hbox{for all }u,v\in\R^m\\
 &\qquad\Longleftrightarrow\quad
 \Phi_C(Y)\succeq\eta CYC\quad\hbox{for all }Y\succeq0.
\end{aligned}
\label{eq:positive-map}
\end{equation}
The implication from rank-one $Y$ to arbitrary $Y\succeq0$ follows
from a rank-one decomposition and linearity.

A second dimension-independent observation separates the centering
bound from the polarized comparison. Put $R=R_C$,
$h=h_C(u,v)$, and $\gamma=u^{\mathrm{T} }Cv$. Since $\one^{\mathrm{T} }h=\gamma$ and
$R\one=\one$,
\begin{equation}
 h^{\mathrm{T} }Rh-\frac{\gamma^2}{m}
 =\left(h-\frac\gamma m\one\right)^{\mathrm{T} }
 R\left(h-\frac\gamma m\one\right)\ge0.
\label{eq:polarized-second}
\end{equation}
Because $R\succ0$, equality holds exactly when
$h=(\gamma/m)\one$.

\begin{samepage}
\subsection{An equivalence in consecutive dimensions}

\begin{theorem}[Consecutive-dimension equivalence]
\label{thm:equivalence}
For an integer $m\ge1$, the following assertions are equivalent:
\begin{enumerate}
\item[(i)] For all $C\in S^m_{++}$ and $u,v\in\R^m$,
\[
 d_C(u)^{\mathrm{T} }R_Cd_C(v)\ge\frac14(u^{\mathrm{T} }Cv)^2.
\]
\item[(ii)] For all $A\in S^{m+1}_{++}$ and $w\in\R^{m+1}$,
$R_Ad_A(w)\ge0$.
\item[(iii)] For all $G\in S^{m+2}_{++}$, $R_G\ge0$.
\item[(iv)] For all $A\in S^{m+1}_{++}$ and $E\succeq0$,
$(K_A+A\circ E)^{-1}\one\le\one$.
\end{enumerate}
\end{theorem}
\end{samepage}

\begin{proof}
We first prove (i)$\Rightarrow$(ii). By normalization and permutation
invariance, it suffices to examine the first coordinate for
$A=\mathcal{B}(C,x)$. Use Lemma~\ref{lem:border}, and put
\[
 X=zz^{\mathrm{T} }+\Diag(t)C^{-1}\Diag(t)\succeq0,\qquad
 S=\Sym(\Diag(t)C)=\Phi_C(X),\qquad r=Cz.
\]
Here $X$ and the identity $t=R_C\diag(CX)$ are supplied by
Lemma~\ref{lem:border}. Applying (i) through
\eqref{eq:positive-map} to this \emph{full} positive
semidefinite matrix $X$ gives
\[
 \Phi_C(X)\succeq\tfrac14 CXC
 =\tfrac14 rr^{\mathrm{T} }
  +\tfrac14 C\Diag(t)C^{-1}\Diag(t)C.
\]
Consequently,
\begin{equation}
 Q_0:=S-\frac14rr^{\mathrm{T} }
 \succeq\frac14C\Diag(t)C^{-1}\Diag(t)C\succeq0.
\label{eq:quadratic-border-bound}
\end{equation}
Write an arbitrary vector as
\[
 w=\begin{pmatrix}\alpha\\y-\alpha x\end{pmatrix},
 \qquad Aw=\begin{pmatrix}\alpha+x^{\mathrm{T} }Cy\\Cy\end{pmatrix}.
\]
Using $x-\Diag(x)t=z$ gives the identity
\begin{align}
 \sigma[R_Ad_A(w)]_1
 &=\alpha^2+\alpha z^{\mathrm{T} }Cy+y^{\mathrm{T} }\Diag(t)Cy\notag\\
 &=\left(\alpha+\frac12r^{\mathrm{T} }y\right)^2+y^{\mathrm{T} }Q_0y\ge0.
\label{eq:border-square}
\end{align}
This proves (ii).

For (ii)$\Rightarrow$(iii), distinguish an arbitrary coordinate of $G$
and normalize it as $G=\mathcal{B}(A,x)$. Lemma~\ref{lem:border}
gives $e_1^{\mathrm{T} }R_G=\sigma^{-1}(1,t^{\mathrm{T} })$ with
$t=R_A\diag(AX)$ and $X\succeq0$. A decomposition
$X=\sum_\ell u_\ell u_\ell^{\mathrm{T} }$ yields
\[
 t=\sum_\ell R_Ad_A(u_\ell)\ge0.
\]
Every row is therefore nonnegative.

Assume (iii), fix $A\in S^{m+1}_{++}$ and $w\in\R^{m+1}$, and
consider $\mathcal{B}(A,\varepsilon w)$. Its border vector is
\[
 t_\varepsilon=\varepsilon^2
  (K_A+\varepsilon^2\Diag(w)A\Diag(w))^{-1}d_A(w).
\]
It is nonnegative by (iii). Divide by $\varepsilon^2$ and let
$\varepsilon\to0$ to obtain (ii).

Finally assume (ii), fix $C,u,v$ in (i), and set
$A_\varepsilon=\mathcal{B}(C,\varepsilon u)$. With the notation
of Lemma~\ref{lem:border}, analytic expansion of the inverse gives
\[
 t_\varepsilon=\varepsilon^2R_Cd_C(u)+O(\varepsilon^4),\quad
 q_\varepsilon=\one+O(\varepsilon^2),\quad
 z_\varepsilon=\varepsilon u+O(\varepsilon^3).
\]
For a fixed real $\alpha$, take
$w_\varepsilon=(\varepsilon\alpha,
(v-\varepsilon^2\alpha u)^{\mathrm{T} })^{\mathrm{T} }$. The first equality in
\eqref{eq:border-square}, without its positivity conclusion, gives
\[
 0\le\sigma_\varepsilon
 [R_{A_\varepsilon}d_{A_\varepsilon}(w_\varepsilon)]_1
 =\varepsilon^2\bigl(\alpha^2+\alpha u^{\mathrm{T} }Cv
                  +d_C(u)^{\mathrm{T} }R_Cd_C(v)\bigr)
   +O(\varepsilon^4).
\]
Division by $\varepsilon^2$ and the choice $\alpha=-u^{\mathrm{T} }Cv/2$
prove (i). The equivalence of (ii) and (iv) is
Proposition~\ref{prop:load-equivalence} and
\eqref{eq:dictionary}, applied to every $A$ in the stated order.
\end{proof}

\begin{remark}
At $C=I_m$ and $u=v=\one$, condition (i) requires $m\ge m^2/4$.
Thus none of the equivalent universal assertions holds for $m\ge5$.
The polarized inequality below establishes all four assertions for
$1\le m\le4$.
The theorem is an equivalence in every dimension; it does not assert
positivity in every dimension.
\end{remark}

The sharp thresholds are summarized below. Each property is required
for every positive definite input of the stated order and, where
applicable, for every choice of vectors. The explicit counterexamples are given in
Section~\ref{sec:analytic}.
\begin{center}
\begin{tabular}{@{}p{0.46\textwidth}cp{0.31\textwidth}@{}}
\toprule
Universal property & \shortstack{Largest\\order} & First failure\\
\midrule
$d_C(u)^{\mathrm{T} }R_Cd_C(v)\ge\tfrac14(u^{\mathrm{T} }Cv)^2$
 & $4$ & $C=I_5$, $u=v=\one$\\[5pt]
$S_{C,i}\succeq0$ for every $i$
 & $5$ & Order six; Example~\ref{ex:auxiliary-six}\\[5pt]
$R_C\ge0$
 & $6$ & Order seven; Example~\ref{ex:seven}\\[3pt]
\bottomrule
\end{tabular}
\end{center}

Corollary~\ref{cor:auxiliary} establishes condition (ii) for
$1\le m\le4$. Theorem~\ref{thm:equivalence} therefore already gives
the weaker biquadratic bound from the analytic result. The stronger
polarized comparison below follows from its explicit polynomial
certificate, without using the continuation argument.

\subsection{The four-dimensional polarized inequality}

\begin{proposition}[Four-dimensional polarized inequality]
\label{prop:polarized}
For $C\in S^4_{++}$ and $u,v\in\R^4$,
\begin{equation}
 d_C(u)^{\mathrm{T} }R_Cd_C(v)
 \ge h_C(u,v)^{\mathrm{T} }R_Ch_C(u,v)
 \ge\frac14(u^{\mathrm{T} }Cv)^2.
\label{eq:polarized}
\end{equation}
Consequently,
\begin{equation}
 \Phi_C(Y)\succeq\frac14CYC\qquad(Y\succeq0).
\label{eq:four-dimensional-map}
\end{equation}
\end{proposition}

\begin{proof}
Write $R=R_C$ and $h=h_C(u,v)$. Equation~\eqref{eq:polarized-second}
with $m=4$ proves the second inequality. We prove the first by expressing
its difference as a quadratic form on bivectors.

Use the basis $(e_i\wedge e_j)_{i<j}$ of $\Lambda^2\R^4$, ordered as
\[
 \mathcal{P}=(12,13,14,23,24,34).
\]
Set $\omega=u\wedge v$, so that
$\omega_{ij}=u_iv_j-u_jv_i$. Define a symmetric matrix
$\mathcal{E}$ indexed by $\mathcal{P}$ by
\begin{equation}
\begin{split}
 4\mathcal{E}_{ij,kl}
 ={}&C_{ik}C_{jl}(R_{ij}+R_{il}+R_{kj}+R_{kl})\\
    &-C_{il}C_{jk}(R_{ij}+R_{ik}+R_{lj}+R_{lk}).
\end{split}
\label{eq:exterior-matrix}
\end{equation}
Then
\begin{equation}
 d_C(u)^{\mathrm{T} }Rd_C(v)-h_C(u,v)^{\mathrm{T} }Rh_C(u,v)
 =\omega^{\mathrm{T} }\mathcal{E}\omega.
\label{eq:polarization-defect}
\end{equation}
To verify this identity, put
$A_s=(\Diag(e_s)C+C\Diag(e_s))/2$.
For symmetric $H,F$, define
\[
 \mathcal{C}(H,F)_{ij,kl}
 =\tfrac12(H_{ik}F_{jl}+F_{ik}H_{jl}
            -H_{il}F_{jk}-F_{il}H_{jk}).
\]
Expansion gives
\[
 \omega^{\mathrm{T} }\mathcal{C}(H,F)\omega
 =\tfrac12\bigl[(u^{\mathrm{T} }Hu)(v^{\mathrm{T} }Fv)+(u^{\mathrm{T} }Fu)(v^{\mathrm{T} }Hv)\bigr]
  -(u^{\mathrm{T} }Hv)(u^{\mathrm{T} }Fv).
\]
Apply this with $H=A_s,F=A_t$ and sum against $R_{st}$.
Since $[d_C(u)]_s=u^{\mathrm{T} }A_su$ and $[h_C(u,v)]_s=u^{\mathrm{T} }A_sv$,
the resulting matrix is $\mathcal E=\sum_{s,t}R_{st}\mathcal C(A_s,A_t)$.
For completeness, the coefficients of this Kulkarni--Nomizu-type
expression can be read off without expanding the quadratic form:
\[
 (A_s)_{ik}=\tfrac12(\delta_{si}+\delta_{sk})C_{ik},\qquad
 \sum_{s,t}R_{st}(A_s)_{ik}(A_t)_{jl}
 =\frac{C_{ik}C_{jl}}4
   (R_{ij}+R_{il}+R_{kj}+R_{kl}).
\]
Because $R$ is symmetric, the two terms in each symmetrized pair
in $\mathcal C(A_s,A_t)$ have equal sums. Consequently,
\[
 \begin{aligned}
  \mathcal E_{ij,kl}
  &=\sum_{s,t}R_{st}\bigl((A_s)_{ik}(A_t)_{jl}
                    -(A_s)_{il}(A_t)_{jk}\bigr)\\
  &=\frac14\bigl[C_{ik}C_{jl}
       (R_{ij}+R_{il}+R_{kj}+R_{kl})\\
  &\hspace{3.5em}{}-C_{il}C_{jk}
       (R_{ij}+R_{ik}+R_{lj}+R_{lk})\bigr],
 \end{aligned}
\]
which is exactly \eqref{eq:exterior-matrix}.

Let $J$ have the only nonzero entries
\[
 J_{12,34}=J_{34,12}=1,\quad
 J_{13,24}=J_{24,13}=-1,\quad
 J_{14,23}=J_{23,14}=1.
\]
Decomposable bivectors satisfy the Pl\"ucker relation
\begin{equation}
 \omega^{\mathrm{T} }J\omega
 =2(\omega_{12}\omega_{34}-\omega_{13}\omega_{24}
                         +\omega_{14}\omega_{23})=0.
\label{eq:plucker}
\end{equation}
For the standard orientation of $\R^4$, $J$ represents the Euclidean
Hodge star in this basis. The representation by $\mathcal{C}(H,F)$
gives the first Bianchi identity
$\mathcal{E}_{12,34}-\mathcal{E}_{13,24}+\mathcal{E}_{14,23}=0$.
Thus the search for a scalar correction
$\mathcal{E}+\tau J\succeq0$ has the Finsler--Thorpe interpretation
\cite[Proposition~2.2]{BettiolKrishnan}: it is equivalent to
nonnegativity on decomposable bivectors. We give an explicit correction
and a polynomial certificate for this particular form.

Put
\[
 s_1=R_{12}+R_{34},\quad s_2=R_{13}+R_{24},\quad
 s_3=R_{14}+R_{23},
\]
and define
\begin{equation}
 \tau=\frac14\bigl[
 C_{12}C_{34}(s_3-s_2)+C_{13}C_{24}(s_1-s_3)
                         +C_{14}C_{23}(s_2-s_1)\bigr].
\label{eq:tau}
\end{equation}
The correction $\tau$ and the ten square vectors below arose from
computer-assisted symbolic exploration; neither their selection nor a
conceptual derivation of them is needed for the proof. What is proved
is the exact identity~\eqref{eq:ten-squares}, together with the
algebraic reductions and congruences below. In particular, the
Finsler--Thorpe criterion motivates searching for a multiple of $J$
but does not itself establish that the specific choice~\eqref{eq:tau}
is positive semidefinite.

We first justify diagonal normalization. Every $C\in S^4_{++}$ can
be written $C=DC_0D$, where $D=\Diag(d_1,\ldots,d_4)$ is positive
diagonal and $C_0=LL^{\mathrm{T} }$ for a unit lower triangular $L$. Indeed,
for a Cholesky factor $T$ take $D=\Diag(T_{11},\ldots,T_{44})$
and $L=D^{-1}T$. The invariance~\eqref{eq:diagonal-invariance}
gives $R_C=R_{C_0}$, and direct substitution yields
\[
 d_C(u)=d_{C_0}(Du),\quad
 h_C(u,v)=h_{C_0}(Du,Dv),\quad
 u^{\mathrm{T} }Cv=(Du)^{\mathrm{T} }C_0(Dv).
\]
Writing $D_2=\Lambda^2D=\Diag(d_id_j)_{i<j}$, we obtain from
\eqref{eq:exterior-matrix} and \eqref{eq:tau}
\[
 \mathcal E_C=D_2\mathcal E_{C_0}D_2,\qquad
 \tau_C=(\det D)\tau_{C_0},\qquad
 D_2JD_2=(\det D)J.
\]
Thus
\[
 \mathcal E_C+\tau_CJ
 =D_2(\mathcal E_{C_0}+\tau_{C_0}J)D_2,
\]
so positive semidefiniteness in normalized coordinates implies the
same claim for arbitrary $C$. It therefore suffices to consider
\begin{equation}
 C=LL^{\mathrm{T} },\qquad
 L=\begin{pmatrix}
 1&0&0&0\\a&1&0&0\\b&c&1&0\\d&e&f&1
 \end{pmatrix},\qquad a,b,c,d,e,f\in\R.
\label{eq:cholesky-four}
\end{equation}
Let $B_2=\Lambda^2(L^{-1})$, characterized by
$B_2(x\wedge y)=(L^{-1}x)\wedge(L^{-1}y)$.
Its entries are the $2\times2$ minors
\[
 (B_2)_{ij,kl}
 =(L^{-1})_{ik}(L^{-1})_{jl}-(L^{-1})_{il}(L^{-1})_{jk},
\]
where, explicitly,
\[
 L^{-1}=\begin{pmatrix}
  1&0&0&0\\ -a&1&0&0\\ ac-b&-c&1&0\\
  ae+bf-acf-d&cf-e&-f&1
 \end{pmatrix}.
\]
For a compact intermediate coefficient check, write
$\Delta=\det K_C>0$, $A=\adj(K_C)$, and
$\vartheta=4\Delta\tau$. Explicitly,
\begin{align}
 \vartheta={}&C_{12}C_{34}(A_{14}+A_{23}-A_{13}-A_{24})\notag\\
 &+C_{13}C_{24}(A_{12}+A_{34}-A_{14}-A_{23})\notag\\
 &+C_{14}C_{23}(A_{13}+A_{24}-A_{12}-A_{34}).
 \label{eq:theta-polynomial}
\end{align}
Then the symmetric $6\times6$ polynomial matrix
$\mathcal T=4\Delta(\mathcal E+\tau J)$ has entries
\begin{equation}
\begin{split}
 \mathcal T_{ij,kl}={}&
 C_{ik}C_{jl}(A_{ij}+A_{il}+A_{kj}+A_{kl})\\
 &-C_{il}C_{jk}(A_{ij}+A_{ik}+A_{lj}+A_{lk})
 +\vartheta J_{ij,kl}.
\end{split}
\label{eq:T-coefficients}
\end{equation}
For instance, the coefficient rule gives the readily checked entries
\begin{align*}
 \mathcal T_{12,12}
 &=4C_{11}C_{22}A_{12}
   -C_{12}^2(A_{11}+A_{22}+2A_{12}),\\
 \mathcal T_{12,13}
 &=2C_{11}C_{23}(A_{12}+A_{13})
   -C_{12}C_{13}(A_{11}+A_{12}+A_{13}+A_{23}),\\
 \mathcal T_{12,34}
 &=C_{13}C_{24}(A_{12}+A_{14}+A_{23}+A_{34})\\
 &\quad-C_{14}C_{23}(A_{12}+A_{13}+A_{24}+A_{34})
   +\vartheta.
\end{align*}
These formulas specify the entire intermediate matrix without
listing 36 expanded polynomials; its other entries follow by
\eqref{eq:T-coefficients} and symmetry. They also make clear how
the denominators disappear: $C^{-1}=L^{-\mathrm{T} }L^{-1}$ and $L^{-1}$
are polynomial in $a,b,c,d,e,f$, while $A$ is the polynomial
adjugate of the polynomial matrix $K_C$.
Define
\begin{equation}
 \mathcal W=B_2\mathcal T B_2^{\mathrm{T} }
 =4\Delta B_2(\mathcal E+\tau J)B_2^{\mathrm{T} }.
\label{eq:W-certificate}
\end{equation}
For the one-parameter specialization $b=c=d=e=f=0$, these
coefficient formulas give
$\mathcal W=2a^2 e_{12}e_{12}^{\mathrm{T} }$, where $e_{12}$ denotes
the first standard coordinate vector in the ordered bivector basis.
This provides a simple normalization check independent of the full
factorization.

The ten vectors $p_\ell$ listed in Appendix~\ref{app:certificate}
satisfy the polynomial identity
\begin{equation}
 \mathcal{W}=\sum_{\ell=1}^8p_\ell p_\ell^{\mathrm{T} }
                  +2p_9p_9^{\mathrm{T} }+2p_{10}p_{10}^{\mathrm{T} }.
\label{eq:ten-squares}
\end{equation}
All coordinates are explicit; this is an identity in the polynomial
ring $\Q[a,b,c,d,e,f]$, not a numerical SOS feasibility result.
It can be checked entrywise by forming $K_C$, $A$, and $\vartheta$,
then computing $\mathcal T$ using~\eqref{eq:T-coefficients},
$B_2\mathcal T B_2^{\mathrm{T} }$, and the ten outer products on the right.
The accompanying exact-arithmetic verifier \texttt{verify\_four\_dimensional\_sos.py}
implements this finite comparison of all 36 matrix entries.
No particular search algorithm for $\tau$ or the $p_\ell$ is
assumed in this verification.

The right side is positive semidefinite for every real choice of the
six Cholesky parameters. Since $\Delta>0$ and $B_2$ is invertible,
the identity $\mathcal W=4\Delta B_2(\mathcal E+\tau J)B_2^{\mathrm{T} }$
is an invertible congruence with a positive scalar factor. Hence
\eqref{eq:ten-squares} proves $\mathcal{E}+\tau J\succeq0$;
this implication uses no decomposability assumption. More explicitly,
for $\xi=B_2^{-\mathrm{T} }\omega$,
\begin{equation}
\begin{split}
 d_C(u)^{\mathrm{T} }Rd_C(v)-h^{\mathrm{T} }Rh
 &=\omega^{\mathrm{T} }(\mathcal{E}+\tau J)\omega\\
 &=\frac{\sum_{\ell=1}^8(p_\ell^{\mathrm{T} }\xi)^2
                  +2(p_9^{\mathrm{T} }\xi)^2+2(p_{10}^{\mathrm{T} }\xi)^2}
        {4\det(K_C)}\ge0.
\end{split}
\label{eq:explicit-polarized-sos}
\end{equation}
Together with \eqref{eq:polarized-second}, this proves
\eqref{eq:polarized}. Equation~\eqref{eq:four-dimensional-map}
follows from \eqref{eq:positive-map}.
\end{proof}

\begin{remark}
The certificate \eqref{eq:ten-squares} is valid for every
$\xi\in\R^6$; the Pl\"ucker relation enters when it is applied to
the polarization defect. If $u,v$ are collinear, then
$\omega=u\wedge v=0$, so the defect vanishes by
\eqref{eq:polarization-defect}.
\end{remark}

\begin{corollary}[Polarization in orders at most four]
\label{cor:polarized-small-orders}
For $1\le n\le4$, $C\in S^n_{++}$, and $u,v\in\R^n$,
\begin{equation}
 d_C(u)^{\mathrm{T} }R_Cd_C(v)
 \ge h_C(u,v)^{\mathrm{T} }R_Ch_C(u,v)
 \ge\frac1n(u^{\mathrm{T} }Cv)^2.
\label{eq:polarized-small-orders}
\end{equation}
The constant $1/n$ is optimal. Equality in the second inequality
holds exactly when
$h_C(u,v)=(u^{\mathrm{T} }Cv)\one/n$.
\end{corollary}

\begin{proof}
For $n<4$, apply Proposition~\ref{prop:polarized} to
$C\oplus I_{4-n}$ and the vectors $(u,0)$ and $(v,0)$.
The direct-sum identities give the first inequality.
Equation~\eqref{eq:polarized-second}, with $m=n$, gives the second
inequality and its equality condition. At $C=I_n$ and $u=v=\one$,
the two expressions on the left equal $n$, while $(u^{\mathrm{T} }Cv)^2=n^2$.
Thus $1/n$ cannot be increased.
\end{proof}

\begin{example}[Diagonal covariances]
If $C=\Diag(c_1,\ldots,c_n)\succ0$, then $R_C=I_n$ and
\[
 d_C(u)^{\mathrm{T} }R_Cd_C(v)
 =h_C(u,v)^{\mathrm{T} }R_Ch_C(u,v)
 =\sum_{i=1}^n c_i^2u_i^2v_i^2.
\]
Hence the polarization defect vanishes identically. The centering
bound is an equality exactly when the products $c_i u_i v_i$ are
all equal. For $C=I_n$, this is the familiar equality condition in
$\sum_i(u_iv_i)^2\ge(\sum_i u_iv_i)^2/n$.
\end{example}

\begin{proof}[Second proof of IRGA nonnegativity through order six]
Proposition~\ref{prop:polarized} establishes condition (i) of
Theorem~\ref{thm:equivalence} for $m=4$. The theorem gives quadratic
positivity in order five and IRGA nonnegativity in order six.
For $A\in S^n_{++}$ with $n<6$, apply the order-six result to
$A\oplus I_{6-n}$, whose IRGA is $R_A\oplus I_{6-n}$.
The analytic proof in Section~\ref{sec:analytic} supplies the exact
support and strictness assertions.
\end{proof}

\section{Exact precision-load identities and polynomial certificates}
\label{sec:loads}

The common load representation supplies the sign needed for bordering.
Here we resolve the load deficit into explicit nonnegative terms and
derive polynomial certificates when the covariance is fixed. The
multipliers in these identities are constructed from an unloaded IRGA
column; the loaded column may have entries of either sign.

\subsection{An exact deficit identity}

\begin{lemma}[Load-deficit identity]
\label{lem:deficit}
Let $C\succ0$ be irreducible, fix $i$, and suppose that
$w=R_Ce_i>0$ and $S_{C,i}\succ0$. For $E\succeq0$, put
\begin{align*}
 D&=\Diag(w),& B&=K_C+C\circ E,&p&=B^{-1}e_i,\\
 P&=\Diag(p),&\delta&=p-w,&
 Q&=\Sym(D^{-1}C),\qquad T=\Sym(D^{-1}K_C).
\end{align*}
Then $B,Q,T\succ0$ and
\begin{equation}
 1-\one^{\mathrm{T} }p=\delta^{\mathrm{T} }T\delta+\tr(PQPE).
\label{eq:deficit-identity}
\end{equation}
In particular, $\one^{\mathrm{T} }p\le1$, with equality if and only if $E=0$.
No entrywise sign is assumed for $p$ or $B^{-1}$.
\end{lemma}

\begin{proof}
The Schur product argument gives $B\succ0$. Writing
$H=\Sym(CD)\succ0$, we have
\begin{equation}
 Q=D^{-1}HD^{-1}\succ0,\qquad
 T=C^{-1}\circ Q\succ0.
\label{eq:deficit-multipliers}
\end{equation}
For the second identity, both sides have $(j,k)$ entry
$\tfrac12(w_j^{-1}+w_k^{-1})C_{jk}(C^{-1})_{jk}$.

Multiply $Bp=e_i$ on the left by $p^{\mathrm{T} }D^{-1}$:
\begin{equation}
 p^{\mathrm{T} }D^{-1}K_Cp+p^{\mathrm{T} }D^{-1}(C\circ E)p=\frac{p_i}{w_i}.
\label{eq:deficit-start}
\end{equation}
Averaging the double sum for the second term with the same sum with
$j,k$ interchanged gives
\[
 p^{\mathrm{T} }D^{-1}(C\circ E)p=\tr(PQPE).
\]
For the first term use $p=w+\delta$, $K_Cw=e_i$,
$w^{\mathrm{T} }D^{-1}=\one^{\mathrm{T} }$, and $\one^{\mathrm{T} }K_C=\one^{\mathrm{T} }$. The four terms yield
\begin{align*}
 p^{\mathrm{T} }D^{-1}K_Cp
 &=1+\one^{\mathrm{T} }\delta+\frac{\delta_i}{w_i}+\delta^{\mathrm{T} }T\delta\\
 &=\one^{\mathrm{T} }p+\frac{\delta_i}{w_i}+\delta^{\mathrm{T} }T\delta.
\end{align*}
Since $p_i/w_i=1+\delta_i/w_i$, cancellation in
\eqref{eq:deficit-start} proves \eqref{eq:deficit-identity}.

The first term on the right is nonnegative. The congruence $PQP$
is positive semidefinite even if $P$ is singular or has negative
entries, so the trace term is also nonnegative. If equality holds,
$T\succ0$ first forces $\delta=0$, hence $p=w$ and $P=D$.
Only then do we conclude
$PQP=DQD=H\succ0$, and $\tr(HE)=0$ forces $E=0$.
Conversely, $E=0$ gives $p=w$ and $\one^{\mathrm{T} }p=1$.
\end{proof}

The hypotheses of Lemma~\ref{lem:deficit} hold for every irreducible
covariance of order at most five by Corollary~\ref{cor:auxiliary}.
The lemma itself has no dimension restriction. Applying it on each
component gives the same equality characterization as
Corollary~\ref{cor:component-load}.

\begin{example}[A loaded inverse need not be nonnegative]
Let
\[
 C=\begin{pmatrix}1&1/4\\1/4&1\end{pmatrix},\qquad
 E=\frac13\begin{pmatrix}1&2\\2&4\end{pmatrix}.
\]
Then
\[
 B=K_C+C\circ E=
 \begin{pmatrix}7/5&1/10\\1/10&12/5\end{pmatrix},\qquad
 B^{-1}e_1=\frac1{67}\begin{pmatrix}48\\-2\end{pmatrix}.
\]
The selected column has a negative entry, while its sum is
$46/67<1$. A column-sum bound cannot be replaced by entrywise
nonnegativity of the loaded inverse.
\end{example}

\begin{corollary}[Normalized positive loaded weights]
\label{cor:normalized-weights}
Let $C\in S^m_{++}$ with $m\le5$ and $E\succeq0$. If
\[
 (K_C+C\circ E)p=e_i,\qquad p>0,\qquad\one^{\mathrm{T} }p=1,
\]
then $C$ is irreducible, $E=0$, and $p=R_Ce_i$.
\end{corollary}

\begin{proof}
If $C$ had a component not containing $i$, its positive definite
loaded block would force the corresponding coordinates of $p$ to
vanish. This contradicts $p>0$. Thus $C$ is irreducible, and
Lemma~\ref{lem:deficit} gives $E=0$.
\end{proof}

An explicit separating certificate also follows. Fix an irreducible
$C$ satisfying the hypotheses of Lemma~\ref{lem:deficit}, and let
$p>0$ with $\one^{\mathrm{T} }p=1$ be a proposed loaded column. Set
$w=R_Ce_i$, $D=\Diag(w)$, and $P=\Diag(p)$, and define $Q,T$
as in that lemma. Put
$y=D^{-1}p$ and $d=e_i-K_Cp$. For
$\Psi(E)=(C\circ E)p$, the trace adjoint satisfies
\[
 \Psi^*(y)=\tfrac12\bigl(\Diag(y)CP+PC\Diag(y)\bigr)=PQP\succ0,
 \qquad
 y^{\mathrm{T} }d=-(p-w)^{\mathrm{T} }T(p-w).
\]
If $p\ne w$, the second expression is negative, whereas a positive
semidefinite solution of $\Psi(E)=d$ would require
$y^{\mathrm{T} }d=\tr(\Psi^*(y)E)\ge0$. If $p=w$, then $d=0$ and only
$E=0$ is possible.

\subsection{A positive formula for an individual IRGA entry}

To retain the scale of an arbitrary bordered covariance, let
$C\in S^m_{++}$ and $b\in\R^m$, and write
\begin{equation}
 G=\begin{pmatrix}C&b\\b^{\mathrm{T} }&d\end{pmatrix}\succ0,
 \qquad s=d-b^{\mathrm{T} }C^{-1}b>0,\qquad v=C^{-1}b.
\label{eq:unnormalized-border}
\end{equation}
Set
\[
 E=vv^{\mathrm{T} }/s,\qquad B=K_C+C\circ E,\qquad
 a=\one-B\one,\qquad \sigma=d/s-a^{\mathrm{T} }B^{-1}a.
\]
Block inversion and the unit row sums of $K_G$ give
\begin{equation}
 K_G=\begin{pmatrix}B&a\\a^{\mathrm{T} }&d/s\end{pmatrix},\qquad
 \sigma>0,\qquad
 \sigma(R_G)_{1:m,m+1}=\one-B^{-1}\one.
\label{eq:entry-from-deficit}
\end{equation}
Suppose first that $C$ is irreducible and satisfies the hypotheses of
Lemma~\ref{lem:deficit} for the selected index $i$. Use the quantities
$w,p,P,Q,T$ from that lemma with $E=vv^{\mathrm{T} }/s$.
Since $B$ is symmetric, $(B^{-1}\one)_i=\one^{\mathrm{T} }p$.
Equations~\eqref{eq:entry-from-deficit} and~\eqref{eq:deficit-identity}
therefore give
\begin{equation}
 \sigma(R_G)_{i,m+1}
 =(p-w)^{\mathrm{T} }T(p-w)+\frac1s(Pv)^{\mathrm{T} }Q(Pv).
\label{eq:finite-entry}
\end{equation}
For reducible $C$, suppose that the component containing $i$ satisfies
the selected-column hypotheses of Lemma~\ref{lem:deficit}. Use the
local quantities on that component and replace $v$ by its restriction
there, retaining the global positive Schur complements $s$ and $\sigma$.
The identity holds at arbitrary finite coupling strength. Its
multipliers depend on the covariance, so it does not by itself give a
polynomial SOS identity in jointly varying covariance and coupling
parameters.

\begin{samepage}
\subsection{Polynomial SOS certificates with the covariance fixed}

\begin{theorem}[Fixed-covariance load numerator]
\label{thm:fixed-sos}
Let $C\succ0$ be a fixed irreducible covariance of order $1\le m\le5$,
fix an index $i$, and let $r\ge1$. Regard the entries of
$Z=[z_1\ \cdots\ z_r]\in\R^{m\times r}$ as independent variables,
and set
\begin{equation}
 B(Z)=K_C+C\circ ZZ^{\mathrm{T} },\qquad h(Z)=\det B(Z),\qquad
 g_i(Z)=h(Z)-\one^{\mathrm{T} }\adj(B(Z))e_i.
\label{eq:fixed-numerator}
\end{equation}
Then $g_i$ is a sum of squares of real polynomials in the entries of
$Z$, and
\begin{equation}
 g_i(Z)=0\quad\Longleftrightarrow\quad Z=0.
\label{eq:fixed-zero-set}
\end{equation}
Each polynomial being squared can be chosen to have total degree at
most $m$. More generally, the same conclusions hold in any order
for a fixed irreducible $C$ and selected $i$ satisfying
$R_Ce_i>0$ and $S_{C,i}\succ0$.
\end{theorem}
\end{samepage}

A full proof is given in Appendix~\ref{app:pfaffian}. The key identity
writes a polynomial matrix with SOS symmetric part as
$\mathcal{R}(Z)^{\mathrm{T} }\mathcal{R}(Z)+\Omega(Z)$, with $\Omega$ skew-symmetric.
Its determinant is a sum of squared Pfaffians. In the present case
that determinant, multiplied by a positive constant, is exactly
$g_i$, with no polynomial denominator or additional polynomial
multiplier. For load-factor width $r$, the construction uses
$s=m(r+1)$ rows in $\mathcal R$ and has at most
\[
 N_{m,r}=\sum_{\substack{0\le j\le m\\j\equiv m\pmod 2}}
          \binom{m(r+1)}{j}
\]
candidate Pfaffian-square summands. For each admissible $j$, one
candidate is indexed by every choice of $j$ rows of $\mathcal R$;
form the corresponding skew block matrix in
\eqref{eq:determinant-sos} and take its Pfaffian. Choices with
$j>m$ give zero because the lower-right block is zero. Some of the
remaining Pfaffians may also vanish identically or yield identical
squares, so $N_{m,r}$ is only an upper bound on the number of
nonzero, distinct square terms. The construction is explicit for
arbitrary $r$, although the candidate count can be large. Here $C$
is fixed throughout.

\begin{example}
For
\[
 C=\begin{pmatrix}1&1/2\\1/2&1\end{pmatrix},\qquad
 i=1,\qquad Z=\begin{pmatrix}x\\y\end{pmatrix},
\]
we have
\[
 K_C=\begin{pmatrix}4/3&-1/3\\-1/3&4/3\end{pmatrix},\qquad
 R_Ce_1=\frac15\begin{pmatrix}4\\1\end{pmatrix}.
\]
The numerator is
\begin{align*}
 g_1(x,y)
 &=\frac43x^2+\frac56xy+\frac13y^2+\frac34x^2y^2\\
 &=\left[\frac2{\sqrt3}\left(x+\frac5{16}y\right)\right]^2
   +\left(\frac{\sqrt{13}}8y\right)^2
   +\left(\frac{\sqrt3}2xy\right)^2.
\end{align*}
Its zero set is exactly $\{(0,0)\}$.
\end{example}

\begin{remark}[Scope of the certificate]
The load rank is unrestricted and no sign condition is imposed on
$B(Z)^{-1}e_i$. Substitution of any real polynomial matrix $Z=Z(x)$
preserves the SOS identity. The covariance must remain fixed: the
coefficients of the squares depend on $C$ and $i$ through constant
positive definite factors. Pointwise positivity of those factors as
$C$ varies does not supply a polynomial SOS factorization in covariance
parameters. Theorem~\ref{thm:fixed-sos} therefore does not assert a
joint-parameter SOS certificate for the original order-six numerator.

Irreducibility is also needed for the zero-set conclusion
\eqref{eq:fixed-zero-set}. For example, if $C=I_2$, $i=1$, and
$Z=(0,z)^{\mathrm{T} }$, then $h(Z)=1+z^2$ and $g_1(Z)=0$ for every $z$.
Loads on another covariance component do not change the selected
column sum.
\end{remark}

\section{Structure, averaging, and reverse majorization}
\label{sec:applications}

Some structural properties of $R_G$ hold in every dimension and must
be distinguished from those requiring entrywise nonnegativity. We
first establish the dimension-independent part and then apply the
sharp positivity theorem.

\subsection{Fixed space and spectral convergence in every dimension}

Let $V_1,\ldots,V_r$ be the covariance components of $G\in S^n_{++}$,
and define
\begin{equation}
 \mathcal{F}_G=\{x\in\R^n:x\text{ is constant on each }V_a\},
 \qquad
 \Pi_G=\sum_{a=1}^r\frac{\one_{V_a}\one_{V_a}^{\mathrm{T} }}{|V_a|}.
\label{eq:fixed-space}
\end{equation}
Here $\one_{V_a}$ is the indicator vector of $V_a$, so $\Pi_G$ is
the orthogonal projection onto $\mathcal{F}_G$.

\begin{lemma}[A commutator identity]
\label{lem:fiedler}
For every $G\in S^n_{++}$ and $x\in\R^n$,
\begin{equation}
 x^{\mathrm{T} }(K_G-I)x=\frac12
 \left\|G^{1/2}\Diag(x)G^{-1/2}
       -G^{-1/2}\Diag(x)G^{1/2}\right\|_F^2.
\label{eq:fiedler-identity}
\end{equation}
Consequently,
\begin{equation}
 K_G\succeq I,\qquad 0\prec R_G\preceq I,\qquad
 \ker(K_G-I)=\ker(I-R_G)=\mathcal{F}_G.
\label{eq:spectral-bounds}
\end{equation}
\end{lemma}

\begin{proof}
Set $A=G^{1/2}\Diag(x)G^{-1/2}$. Cyclicity of the trace gives
\[
 \|A\|_F^2=\tr(\Diag(x)G\Diag(x)G^{-1})=x^{\mathrm{T} }K_Gx,
 \qquad \tr(A^2)=\tr(\Diag(x)^2)=x^{\mathrm{T} }x.
\]
Since $\|A-A^{\mathrm{T} }\|_F^2=2\|A\|_F^2-2\tr(A^2)$,
\eqref{eq:fiedler-identity} follows. Its right side vanishes exactly
when $G\Diag(x)=\Diag(x)G$, equivalently
$G_{ij}(x_i-x_j)=0$ for every $i,j$. This is precisely
$x\in\mathcal{F}_G$. Inversion gives the spectral inequalities,
and $I-R_G=R_G(K_G-I)$ identifies the kernels.
\end{proof}

The inequalities in \eqref{eq:spectral-bounds} are due to Fiedler;
see Johnson \cite{Johnson1978}. The identity additionally identifies
the equality space directly from the covariance graph.

\begin{corollary}[Spectral convergence]
\label{cor:convergence}
For every $G\in S^n_{++}$, put $\rho=\|R_G-\Pi_G\|_2$.
Then $\rho<1$, and for $y\in\R^n$ and $k\ge1$,
\begin{equation}
 R_G^k\longrightarrow\Pi_G,\qquad
 \|R_G^ky-\Pi_Gy\|_2\le\rho^k\|y-\Pi_Gy\|_2.
\label{eq:convergence}
\end{equation}
\end{corollary}

\begin{proof}
On $\mathcal{F}_G$, $R_G$ is the identity; on its orthogonal
complement its eigenvalues lie in $(0,1)$. Apply the spectral
theorem. If $\mathcal{F}_G=\R^n$, then $R_G=\Pi_G=I$ and $\rho=0$.
\end{proof}

\begin{corollary}[Eventual double stochasticity]
\label{cor:eventual-stochasticity}
For every $G\in S^n_{++}$, all sufficiently large integer powers
$R_G^k$ are doubly stochastic, with
\[
 (R_G^k)_{ij}>0
 \quad\Longleftrightarrow\quad
 i,j\text{ belong to the same covariance component of }G.
\]
More precisely, with $\rho$ as in Corollary~\ref{cor:convergence}
and $N=\max_a|V_a|$, this conclusion holds for every integer
$k\ge1$ satisfying $\rho^k<1/N$.
\end{corollary}

\begin{proof}
Every power has unit row and column sums and preserves the covariance
block decomposition. The spectral decomposition used in
Corollary~\ref{cor:convergence} gives
$\|R_G^k-\Pi_G\|_2=\rho^k$. Hence, for $i,j\in V_a$,
\[
 (R_G^k)_{ij}\ge\frac1{|V_a|}-\rho^k
              \ge\frac1N-\rho^k>0
\]
whenever $\rho^k<1/N$. Entries between distinct components remain
zero. Since $0\le\rho<1$, the bound holds for all sufficiently large
$k$. If $\rho=0$, it holds for every $k\ge1$; in particular, this
covers diagonal $G$, for which $R_G=\Pi_G=I$.
\end{proof}

Neither the fixed-space identity nor convergence of powers requires
entrywise nonnegativity. The same spectral conclusion holds in
dimensions above six, even when $R_G$ is not a stochastic matrix.

\subsection{Stochastic structure in the sharp dimension range}

For $x,y\in\R^n$, write $x\prec y$ when
\[
 \sum_{i=1}^k x_i^{\downarrow}\le\sum_{i=1}^k y_i^{\downarrow}
 \quad(1\le k<n),\qquad
 \sum_{i=1}^n x_i=\sum_{i=1}^n y_i,
\]
where $\downarrow$ denotes decreasing rearrangement. Equivalently,
$x=Ty$ for some doubly stochastic $T$. We write $\mathcal{B}_n$
for the Birkhoff polytope. Its elements are precisely the convex
combinations of permutation matrices by the classical
Birkhoff--von Neumann theorem \cite{Birkhoff1946,VonNeumann1953}.
The majorization characterization is recalled in \cite{Horn1954}.

\begin{theorem}[Structure of the inverse interaction operator]
\label{thm:structure}
Let $G\in S^n_{++}$ with $1\le n\le6$, and write $K=K_G$ and
$R=R_G$. Then
\begin{equation}
 R\in\mathcal{B}_n\cap S^n_{++},\qquad
 \Pi_G\preceq R\preceq I,
 \qquad
 R=\sum_\pi\theta_\pi P_\pi,
 \quad\theta_\pi\ge0,\quad\sum_\pi\theta_\pi=1,
\label{eq:birkhoff}
\end{equation}
where $P_\pi$ are permutation matrices. The eigenvalue $1$ has
multiplicity $r$ in both $K$ and $R$, where $r$ is the number of
covariance components. Each component block of $R$ is entrywise
strictly positive, and all cross-component entries vanish.

For $y\in\R^n$, the solution $z$ of $Kz=y$ satisfies
\begin{equation}
 z_i=\sum_jR_{ij}y_j,\qquad
 \min_{j\in V_a}y_j\le z_i\le\max_{j\in V_a}y_j\quad(i\in V_a),
 \qquad z\prec y.
\label{eq:averaging-bounds}
\end{equation}
Both coordinate bounds are strict whenever $y$ is nonconstant on
$V_a$. For every real-valued convex function $\varphi$ on an interval
containing the coordinates of $y$,
\begin{equation}
 \sum_i\varphi(z_i)\le\sum_i\varphi(y_i),\qquad
 \|z\|_p\le\|y\|_p\quad(1\le p\le\infty).
\label{eq:convex-contraction}
\end{equation}
For strictly convex $\varphi$, equality in the convex-sum inequality
holds exactly when $y\in\mathcal{F}_G$. The same condition
characterizes equality in the norm inequality for every $1<p<\infty$.
\end{theorem}

\begin{proof}
Theorem~\ref{thm:main} gives double stochasticity and the strict block
support. The Birkhoff--von Neumann theorem
\cite{Birkhoff1946,VonNeumann1953} yields the decomposition
in \eqref{eq:birkhoff}. Lemma~\ref{lem:fiedler} gives the spectral
bounds and the multiplicity of $1$. The decomposition into
$\mathcal{F}_G\oplus\mathcal{F}_G^\perp$ also gives
$\Pi_G\preceq R$.

Each row of $R$ forms a convex combination with strictly positive
weights throughout its component. This proves the coordinate bounds
and their strictness when $y$ is nonconstant on the component.
Double stochasticity gives $Ry\prec y$. Applying Jensen's inequality to each row and then using
the column sums proves the convex-sum inequality. The choices
$\varphi(s)=|s|^p$ give the finite-$p$ norm bounds, and the coordinate
formula gives the $p=\infty$ case. For strictly convex $\varphi$,
equality in a row occurs exactly when all coordinates with positive
row weights agree. By the strict component support, equality in all
rows is equivalent to $y\in\mathcal{F}_G$. Conversely, such vectors
are fixed by $R$. The strict convexity of $\varphi(s)=|s|^p$ for
$1<p<\infty$ proves the norm equality statement.
\end{proof}

Symmetry makes $R_G$ a reversible Markov transition matrix with uniform
stationary distribution. Corollary~\ref{cor:convergence} is then
averaging on each covariance component. If the covariance graph is
connected, the limit is $n^{-1}\one\one^{\mathrm{T} }$. These are statements
about the static interaction equation $K_Gz=y$; they do not assert
closed-loop stability without assumptions on system dynamics.

The associated Dirichlet form is
\begin{equation}
 x^{\mathrm{T} }(I-R_G)x=\frac12\sum_{i,j}(R_G)_{ij}(x_i-x_j)^2.
\label{eq:dirichlet-form}
\end{equation}
Expansion of the squares, symmetry, and the unit row and column sums
prove this identity in every dimension. When $n\le6$, each term on
the right is nonnegative, and the exact support in
Theorem~\ref{thm:main} shows that the form vanishes exactly when
$x\in\mathcal{F}_G$, that is, when $x$ is constant on each covariance
component.

The result also gives the range inclusion
\[
 \{G\circ G^{-1}:G\in S^n_{++}\}
 \subseteq\{K\in S^n:K\succeq I,\ K\one=\one,\ K^{-1}\ge0\},
 \qquad n\le6.
\]
No converse is asserted. Inverse nonnegativity does not require
all off-diagonal entries of $K_G$ to be nonpositive.

\subsection{A reverse Schur--Horn inequality}

The Schur--Horn theorem states that the diagonal of a real symmetric
matrix is majorized by its eigenvalue vector \cite{Horn1954}.
The direction reverses for matrices diagonalized by a real symmetric
positive definite matrix of order at most six. The positive
definiteness assumption concerns the diagonalizer, not the generally
nonsymmetric matrix being diagonalized.

\begin{theorem}[Reverse Schur--Horn]
\label{thm:reverse-schur-horn}
Let $P\in S^n_{++}$, $1\le n\le6$, and $\lambda\in\R^n$.
Set $B=P\Diag(\lambda)P^{-1}$ and $b=\diag(B)$. Then
\begin{equation}
 \lambda=R_Pb,\qquad\lambda\prec b.
\label{eq:reverse-majorization}
\end{equation}
Consequently,
\begin{equation}
 \sum_i\varphi(\lambda_i)\le\sum_i\varphi(b_i),\qquad
 \min_i b_i\le\lambda_j\le\max_i b_i
\label{eq:reverse-convex}
\end{equation}
for every real-valued convex $\varphi$ on an interval containing the
coordinates of $b$. For strictly convex $\varphi$, equality holds
if and only if
\begin{equation}
 P\Diag(\lambda)=\Diag(\lambda)P,
\label{eq:majorization-equality}
\end{equation}
equivalently $B=\Diag(\lambda)$. This is also equivalent to
$\lambda$ and $b$ agreeing up to permutation. The universal
dimension bound is sharp.
\end{theorem}

\begin{proof}
Since $P^{-1}$ is symmetric, direct calculation gives
\begin{equation}
 b_i=\sum_j P_{ij}\lambda_j(P^{-1})_{ji},\qquad b=K_P\lambda.
\label{eq:diagonal-similarity}
\end{equation}
Thus $\lambda=R_Pb$, and Theorem~\ref{thm:structure} proves the
majorization and convexity conclusions. Equality for strictly
convex $\varphi$ is equivalent to $b\in\mathcal{F}_P$. In that
case $\lambda=R_Pb=b\in\mathcal{F}_P$, which is equivalent to
\eqref{eq:majorization-equality} by Lemma~\ref{lem:fiedler}.
The converse is immediate. If $\lambda$ is a permutation of $b$,
their squared Euclidean norms agree; the equality result for
$\varphi(s)=s^2$ gives the same commutation condition.
Sharpness follows from the example after
Proposition~\ref{prop:fixed-majorization} below.
\end{proof}

The term reverse Schur--Horn refers to this majorization inequality.
No assertion is made that every pair $\lambda\prec b$ can be
realized by a positive definite similarity. This is the class called
special PD-diagonalizable in \cite{Uhlmann2023}, where the result
was established through order four.

The conclusion also holds under diagonal scalings and permutations
of the diagonalizer. If
$T=D_1UPVD_2$, where $D_1,D_2$ are nonsingular real diagonal matrices
and $U,V$ are permutation matrices, then
\[
 T\circ T^{-\mathrm{T} }=UK_PV,\qquad
 (T\circ T^{-\mathrm{T} })^{-1}=V^{\mathrm{T} }R_PU^{\mathrm{T} }\in\mathcal{B}_n.
\]
Hence $T\Diag(\lambda)T^{-1}$ has diagonal majorizing $\lambda$,
even though the inverse RGA of $T$ need not be symmetric.

\begin{proposition}[A fixed-diagonalizer characterization]
\label{prop:fixed-majorization}
For any $n\ge1$ and fixed $P\in S^n_{++}$, the following are
equivalent:
\begin{enumerate}
\item[(i)] $R_P\ge0$;
\item[(ii)] $\lambda\prec\diag(P\Diag(\lambda)P^{-1})$ for every
$\lambda\in\R^n$.
\end{enumerate}
\end{proposition}

\begin{proof}
If $R_P\ge0$, its unit sums make it doubly stochastic, and
\eqref{eq:diagonal-similarity} proves (ii). Conversely, take
$\lambda=R_Pe_j$, so that $\diag(P\Diag(\lambda)P^{-1})=e_j$.
For $n\ge2$, the majorization inequality at $k=n-1$, together
with the total sum one, gives $\min_i(R_Pe_j)_i\ge0$.
Thus every column is nonnegative. The scalar case is immediate.
\end{proof}

For $P=G_7$ from Example~\ref{ex:seven}, direct inversion gives
\begin{equation}
 \lambda=R_Pe_2=\frac1{1155}(-1,881,55,55,55,55,55)^{\mathrm{T} },
 \qquad \diag(P\Diag(\lambda)P^{-1})=e_2.
\label{eq:majorization-counterexample}
\end{equation}
The sum of the six largest coordinates of $\lambda$ is
$1156/1155>1$, whereas the corresponding sum for $e_2$ is one.
Thus $\lambda\not\prec e_2$. Adjoining an identity block to $P$
and zero coordinates to $\lambda$ gives failure for every $n\ge7$.
We therefore obtain the sharp classification
\begin{equation}
\begin{split}
 R_G\ge0\text{ for every }G\in S^n_{++}
 \quad&\Longleftrightarrow\quad
 \lambda\prec\diag(P\Diag(\lambda)P^{-1})
 \text{ for every }P\in S^n_{++},\ \lambda\in\R^n\\
 &\Longleftrightarrow\quad n\le6.
\end{split}
\label{eq:classification}
\end{equation}

\subsection{Structured families in higher dimensions}

The universal cutoff does not preclude higher-dimensional structured
families with inverse nonnegativity. For positive definite $A,B$,
\[
 R_{A\oplus B}=R_A\oplus R_B,\qquad
 R_{A\otimes B}=R_A\otimes R_B.
\]
These identities follow from the corresponding inverse and Hadamard
product identities. Hence direct sums and Kronecker products of
positive definite factors of order at most six retain inverse
nonnegativity, and Proposition~\ref{prop:fixed-majorization} gives
reverse majorization for every real spectrum. This extends the
Kronecker construction in \cite{UhlmannPartial} to arbitrary
positive definite factors through order six.

Positive definite matrices with nonpositive off-diagonal entries form
another such family in every dimension. To see the inverse sign
directly, choose $s\ge\max_iG_{ii}$ and write $G=sI-B$ with
$B\ge0$. The Perron eigenvalue of the symmetric nonnegative matrix
$B$ is $s-\lambda_{\min}(G)<s$, so the Neumann series gives
$G^{-1}\ge0$. Consequently $K_G$ is also positive definite with
nonpositive off-diagonal entries. Applying the same argument to
$K_G$ gives $R_G\ge0$.

\section{Concluding remarks}
\label{sec:discussion}

\subsection{Structural reduction and verification}

The two proofs illustrate different ways of changing a certification
problem. The analytic route replaces the expanded polynomial by column
forms and a trace obstruction. Its continuation argument uses graph
connectivity to obtain a positive column at a hypothetical boundary;
the contradiction establishes the auxiliary property through order five.
Bordering then gives positivity through order six and identifies its
exact support. The algebraic route retains certification after changing
the object: the polarization defect becomes a quadratic form on
$\Lambda^2\R^4$, and a Pl\"ucker correction gives a polynomial matrix
with an explicit SOS factorization.

The verification objects are correspondingly concrete. The analytic
argument uses the load identity, the column trace identities, the
singular-boundary trace bound, and continuity. The algebraic argument
additionally uses the identity~\eqref{eq:ten-squares}. Appendix~\ref{app:certificate}
specifies every coefficient, and the accompanying \texttt{verify\_four\_dimensional\_sos.py}
reconstructs both sides and checks all 36 entries by exact polynomial
arithmetic. No numerical tolerance or semidefinite solver enters this
check. The sharpness examples can be verified by rational block
inversion. These finite checks complement the mathematical arguments;
the paper does not present a proof-assistant formalization.

The size reduction should be understood in the appropriate variables.
The covariance in the algebraic proof has order four and is parametrized
by six Cholesky entries. Its certificate is quadratic in six bivector
coordinates. It is a different polynomial object from the original
order-six numerator in 15 variables. The analytic proof avoids that
polynomial expansion altogether.

\subsection{Scope of the methodological conclusion}

ChatGPT assisted in mathematical exploration and manuscript preparation.
The useful intermediate objects were identities and positivity
properties that linked consecutive dimensions. Their discovery changed
the proof task and also produced exact support, load equality cases,
and majorization consequences. This suggests a practical research
strategy: when a generic formulation becomes prohibitive, investigate
invariants, equality conditions, and lower-dimensional statements that
may expose a more tractable representation. Such a strategy is relevant
to specialized applied problems whose useful properties depend on a
fixed dimension or a restricted class of inputs.

The result is a mathematical case study. It does not estimate how often
AI-assisted exploration produces such reductions, compare current
generic solvers, or establish that the proof required AI assistance.
Nor does it settle whether the original polynomial has a joint-parameter
SOS decomposition. Those are separate questions. What the proofs show
is that direct certification of the original large polynomial was
unnecessary: structural reasoning supplies an analytic route and a
smaller object for exact algebraic certification.

\subsection{Further mathematical directions}

The Finsler--Thorpe framework explains why a scalar Pl\"ucker correction
is natural in dimension four. A conceptual derivation of the particular
coefficient~\eqref{eq:tau} and the ten-square factorization would clarify
which special features of the IRGA produce
Proposition~\ref{prop:polarized}. The analytic proof yields the weaker
four-dimensional inequality through the consecutive-dimension
equivalence, but it does not explain the full polarized statement.

The fixed-covariance SOS theorem raises a different question: under what
additional structure can a positive factorization retain polynomial
dependence on covariance parameters? The present proof uses constant
positive definite factors after fixing the covariance. Finally, the load
and bordering framework may identify further families with inverse
nonnegativity in higher dimensions, beyond direct sums, Kronecker
products, and positive definite matrices with nonpositive off-diagonal
entries. The eventual stochasticity of powers holds in every dimension
and raises the question of sharper iteration bounds, particularly
bounds expressed directly in terms of the covariance parameters.

\section*{Acknowledgements}
The authors acknowledge the use of ChatGPT, including GPT-6 Astra, in
mathematical exploration and manuscript preparation. Responsibility for
the mathematical content and final presentation rests with the authors.

\section*{Funding}
This work received partial support from the National Key R\&D Program
of China (No. 2023YFA1009401) and the National Natural Science
Foundation of China (No. 12171324).

\clearpage
\appendix

\section{The ten-square polynomial certificate}
\label{app:certificate}

The vectors in \eqref{eq:ten-squares} have coordinates ordered by
$\mathcal{P}=(12,13,14,23,24,34)$. The parameters are those in
\eqref{eq:cholesky-four}. For arbitrary $\xi\in\R^6$, set
$q_\ell=p_\ell^{\mathrm{T} }\xi$. Then the identity is equivalently
\[
 \xi^{\mathrm{T} }\mathcal{W}\xi=\sum_{\ell=1}^8q_\ell^2+2q_9^2+2q_{10}^2.
\]
No Pl\"ucker relation is imposed on $\xi$. Two of the vectors admit
the shorter definitions
\[
 p_2=2a\,p_9,\qquad p_6=f(p_7-p_8).
\]
The remaining eight vectors are listed below. These definitions
preserve all ten terms in \eqref{eq:ten-squares}. Long coordinates
are split over lines only for readability.

For $p_1$,
\begin{align*}
 p_{1,12}&=2abe^2+ab-2ac^2df+2bcdf-2bde,\\
 p_{1,13}&=2af(be-cd),\\
 p_{1,14}&=abe-acd-bd,\\
 p_{1,23}&=(ac-b)(2cf^2+c-2ef),\\
 p_{1,24}&=2abd+ac^2f-ace-bcf+2be,\\
 p_{1,34}&=2ac^2e+2acf-2bce-bf.
\end{align*}
For $p_3$,
\begin{align*}
 p_{3,12}&=abf+2ac^2d+ad-2bcd,\\
 p_{3,13}&=-abe+acd-bd,\\
 p_{3,14}&=0,\\
 p_{3,23}&=2abd+2ac^2f^3+ac^2f-4acef^2-ace+2ae^2f\\
          &\quad-2bcf^3-bcf+2bef^2+2cdf^2+2cd-2def,\\
 p_{3,24}&=ac^2f^2-2acef+ae^2-bcf^2+2bef-de,\\
 p_{3,34}&=2ac^2ef-2ace^2+2acf^2-2aef-2bcef-bf^2+2cde+df.
\end{align*}
For $p_4$,
\begin{align*}
 p_{4,12}&=c(2acef-2ae^2-a-2bef+2de),\\
 p_{4,13}&=c(2acf^2+ac-2aef-2bf^2-b+2df),\\
 p_{4,14}&=c(acf-ae-bf+2d),\\
 p_{4,23}&=0,\\
 p_{4,24}&=-c(2ad+e),\\
 p_{4,34}&=c(2acd-2bd-f).
\end{align*}
For $p_5$,
\begin{align*}
 p_{5,12}&=2ac^2e+acf+ae-2bce,\\
 p_{5,13}&=-2ac^2f^3-ac^2f+4acef^2+ace-2ae^2f\\
          &\quad+2bcf^3+bcf-2bef^2-2be-2cdf^2+2def,\\
 p_{5,14}&=-ac^2f^2+2acef-ae^2+bcf^2-2cdf+de,\\
 p_{5,23}&=e(2ab+c),\\
 p_{5,24}&=-2af(be-cd),\\
 p_{5,34}&=2abe^2-2ac^2df+2bcdf-2bde+cf^2-ef.
\end{align*}
For $p_7$,
\begin{align*}
 p_{7,12}&=2ac^2f^2+2ac^2-2acef+a-2bcf^2-2bc+bef+cdf,\\
 p_{7,13}&=-ac^2ef+ace^2+acf^2+ac-aef+bcef-b-cde,\\
 p_{7,14}&=-ac^2e-ae+bce+bf,\\
 p_{7,23}&=-abe^2-abf^2+ab+ac^2df+adf-bcdf+bde+c,\\
 p_{7,24}&=-2abf+ac^2d+ad-bcd-cf,\\
 p_{7,34}&=2abe-bd+ce.
\end{align*}
For $p_8$,
\begin{align*}
 p_{8,12}&=2acef-2ae^2-a-bef-cdf+2de,\\
 p_{8,13}&=-ac^2ef+ace^2+acf^2+ac-aef+bcef-cde,\\
 p_{8,14}&=-ac^2e-ae+bce+d,\\
 p_{8,23}&=-abe^2-abf^2-ab+ac^2df+adf-bcdf+bde,\\
 p_{8,24}&=ac^2d-ad-bcd-e,\\
 p_{8,34}&=2acd-bd+ce.
\end{align*}
For $p_9$,
\begin{align*}
 p_{9,12}&=f(be-cd),\\
 p_{9,13}&=\frac12(2bf^2+b-2df),\\
 p_{9,14}&=\frac12(bf-d),\\
 p_{9,23}&=\frac12(2cf^2+c-2ef),\\
 p_{9,24}&=\frac12(cf-e),\\
 p_{9,34}&=bd+ce+f.
\end{align*}
For $p_{10}$,
\begin{align*}
 p_{10,12}&=0,\\
 p_{10,13}&=\tfrac12\bigl(2ac^2ef-2ace^2+2acf^2-2aef\\
           &\hspace{25mm}-2bcef-2bf^2-b+2cde+2df\bigr),\\
 p_{10,14}&=\tfrac12(2ac^2e+2acf-2bce-bf+d),\\
 p_{10,23}&=\tfrac12\bigl(2abe^2+2abf^2+2ab-2ac^2df-2adf\\
           &\hspace{25mm}+2bcdf-2bde+2cf^2+c-2ef\bigr),\\
 p_{10,24}&=-\tfrac12(2ac^2d+2ad-2bcd-cf+e),\\
 p_{10,34}&=0.
\end{align*}

The files \texttt{tau\_sos\_polynomials.json} and
\texttt{tau\_sos\_polynomials.txt} contain the same coefficients.
The script \texttt{verify\_four\_dimensional\_sos.py} constructs
$C$, $K_C$, and $\adj(K_C)$ over $\Q[a,b,c,d,e,f]$ and checks all
36 entries of \eqref{eq:ten-squares}. It also verifies the adjugate
identity, the linearity of each $q_\ell$ in $\xi$, and the positivity
of the rational weights.
They can be downloaded at \url{https://wangjie212.github.io/jiewang/codes/IRGA.zip}.

\section{Pfaffian construction of the fixed-covariance SOS certificate}
\label{app:pfaffian}

We prove Theorem~\ref{thm:fixed-sos} under its stated selected-column
hypotheses. A real polynomial is SOS when it is a finite
sum of squares of real polynomials. A polynomial matrix has an SOS
symmetric part if that symmetric part equals $\mathcal{R}^{\mathrm{T} }\mathcal{R}$
for a rectangular real polynomial matrix $\mathcal{R}$. Pointwise
positive semidefiniteness alone is not the definition used here.

\subsection{Skew-symmetric determinants}

For a skew-symmetric matrix $S$ of order $2k$, define
\begin{equation}
 \Pf(S)=\frac1{2^k k!}
 \sum_{\pi\in\mathfrak{S}_{2k}}\operatorname{sgn}(\pi)
       \prod_{j=1}^k S_{\pi(2j-1),\pi(2j)},\qquad
 \Pf(\varnothing)=1.
\label{eq:pfaffian-definition}
\end{equation}
Equivalently, this is the signed sum over pairings of the indices.
The Pfaffian is a polynomial in the entries; see \cite{Knuth} for
background.

\begin{lemma}
\label{lem:pfaffian-square}
For a real skew-symmetric matrix $S$ of even order,
$\det S=\Pf(S)^2$. In odd order, $\det S=0$. The same statements
hold as polynomial identities for polynomial entries.
\end{lemma}

\begin{proof}
For a real $2k\times2k$ matrix $V$, the definition gives
\begin{equation}
 \Pf(V^{\mathrm{T} }SV)=\det(V)\Pf(S).
\label{eq:pfaffian-congruence}
\end{equation}
Indeed, the left side is multilinear and alternating in the columns
of $V$: each column occurs exactly once in each paired product,
and interchanging columns reverses the sign. Expanding the columns
in the standard basis proves \eqref{eq:pfaffian-congruence}.

We prove the determinant identity by induction on $k$. The empty
case is $1=1$, and a zero matrix of positive even order is immediate.
Otherwise permute the indices so that $a=S_{12}\ne0$. For $j>2$,
replace $e_j$ by
\[
 v_j=e_j+\frac{S_{2j}}a e_1-\frac{S_{1j}}a e_2,
\]
leaving $e_1,e_2$ unchanged. The resulting matrix $V$ has determinant
one, and $e_1^{\mathrm{T} }Sv_j=e_2^{\mathrm{T} }Sv_j=0$. Hence
\[
 V^{\mathrm{T} }SV=\begin{pmatrix}0&a\\-a&0\end{pmatrix}\oplus S'
\]
for a skew-symmetric $S'$ of order $2k-2$. Its determinant is
$a^2\det S'$, and its Pfaffian is $a\Pf(S')$ because the first two
indices must pair with each other. The induction hypothesis and
\eqref{eq:pfaffian-congruence} prove the result. In odd order $d$,
$\det S=\det S^{\mathrm{T} }=\det(-S)=(-1)^d\det S$, so the determinant
vanishes. Applying these identities at all real values of polynomial
entries proves the polynomial identities. The pivots used in the
numeric argument introduce no denominators into the Pfaffian
polynomial itself.
\end{proof}

\begin{lemma}[A determinant-to-SOS identity]
\label{lem:determinant-sos}
Let $\mathcal{R}$ be an $s\times m$ real polynomial matrix and
$\Omega$ an $m\times m$ real polynomial matrix with
$\Omega^{\mathrm{T} }=-\Omega$. For $J\subseteq\{1,\ldots,s\}$, let
$\mathcal{R}_J$ consist of the rows indexed by $J$, in increasing
order. Then
\begin{equation}
 \det(\mathcal{R}^{\mathrm{T} }\mathcal{R}+\Omega)
 =\sum_{\substack{J\subseteq\{1,\ldots,s\}\\m+|J|\text{ even}}}
 \Pf\begin{pmatrix}
       \Omega&\mathcal{R}_J^{\mathrm{T} }\\-\mathcal{R}_J&0
     \end{pmatrix}^{\!2}.
\label{eq:determinant-sos}
\end{equation}
Thus the determinant of a real polynomial matrix with an SOS
symmetric part is an SOS polynomial. No invertibility is required.
\end{lemma}

\begin{proof}
Subtracting $\mathcal{R}^{\mathrm{T} }$ times the lower block row from the upper
block row gives
\[
 \det\begin{pmatrix}\Omega&\mathcal{R}^{\mathrm{T} }\\-\mathcal{R}&I_s\end{pmatrix}
 =\det\begin{pmatrix}
    \Omega+\mathcal{R}^{\mathrm{T} }\mathcal{R}&0\\-\mathcal{R}&I_s
   \end{pmatrix}
 =\det(\Omega+\mathcal{R}^{\mathrm{T} }\mathcal{R}).
\]
Expand the determinant on the left by multilinearity in the added
lower diagonal entries. Selecting an identity entry deletes the
same row and column, with positive cofactor sign. The surviving
indices are all $m$ upper indices and an arbitrary subset $J$ of
the lower indices. Thus the determinant is
\[
 \sum_{J\subseteq\{1,\ldots,s\}}
 \det\begin{pmatrix}\Omega&\mathcal{R}_J^{\mathrm{T} }\\-\mathcal{R}_J&0\end{pmatrix}.
\]
Each matrix is skew-symmetric. Lemma~\ref{lem:pfaffian-square}
turns every even-order term into a square and every odd-order term
into zero. Finally, any polynomial matrix $M$ with
$\Sym(M)=\mathcal{R}^{\mathrm{T} }\mathcal{R}$ has the required decomposition
with $\Omega=(M-M^{\mathrm{T} })/2$.
\end{proof}

\subsection{Proof of the fixed-covariance theorem}

\begin{proof}[Proof of Theorem~\ref{thm:fixed-sos}]
We prove the dimension-free assertion under the selected-column
hypotheses $w=R_Ce_i>0$ and $S_{C,i}\succ0$. The order bound
in the theorem supplies these hypotheses by
Corollary~\ref{cor:auxiliary}. Put
\[
 D=\Diag(w),\qquad Q=\Sym(D^{-1}C)\succ0,\qquad
 T=\Sym(D^{-1}K_C)\succ0,\qquad T_0=I-w\one^{\mathrm{T} }.
\]
These are constant matrices because $C$ and $i$ are fixed. The
identities $K_Cw=e_i$, $D^{-1}w=\one$, and $K_C\one=\one$ give
\[
 Tw=\frac12\left(\frac{e_i}{w_i}+\one\right),\qquad
 w^{\mathrm{T} }Tw=1.
\]
Expanding $T_0^{\mathrm{T} }TT_0$ yields
\begin{equation}
 T_0^{\mathrm{T} }TT_0
 =T-\frac{\one e_i^{\mathrm{T} }+e_i\one^{\mathrm{T} }}{2w_i}
 =\Sym\bigl(D^{-1}(K_C-e_i\one^{\mathrm{T} })\bigr).
\label{eq:fixed-projection}
\end{equation}
Define
\[
 M(Z)=D^{-1}(B(Z)-e_i\one^{\mathrm{T} }).
\]
Since $C\circ ZZ^{\mathrm{T} }=\sum_{\ell=1}^r\Diag(z_\ell)C\Diag(z_\ell)$,
and diagonal matrices commute, \eqref{eq:fixed-projection} gives
\begin{equation}
 \Sym M(Z)=T_0^{\mathrm{T} }TT_0+
              \sum_{\ell=1}^r\Diag(z_\ell)Q\Diag(z_\ell).
\label{eq:fixed-sos-symmetric}
\end{equation}
Choose constant real factors $U^{\mathrm{T} }U=T$ and $V^{\mathrm{T} }V=Q$, and put
\begin{equation}
 \mathcal{R}(Z)=\begin{pmatrix}
 UT_0\\V\Diag(z_1)\\\vdots\\V\Diag(z_r)
 \end{pmatrix},\qquad
 \Omega(Z)=\frac{M(Z)-M(Z)^{\mathrm{T} }}{2}.
\label{eq:fixed-factors}
\end{equation}
Then $\mathcal{R}$ has $s=m(r+1)$ rows and entries of degree at
most one, while $\Omega$ is skew-symmetric with entries of degree
at most two. By Lemma~\ref{lem:determinant-sos}, $\det M(Z)$ is SOS.

It remains to identify the determinant. For every real $Z$,
$B(Z)\succ0$, so $h(Z)>0$. With $p=B(Z)^{-1}e_i$, the rank-one
determinant identity gives
\begin{align}
 \det(B(Z)-e_i\one^{\mathrm{T} })
 &=h(Z)\det(I-p\one^{\mathrm{T} })\notag\\
 &=h(Z)(1-\one^{\mathrm{T} }p)
  =h(Z)-\one^{\mathrm{T} }\adj(B(Z))e_i=g_i(Z).
\label{eq:fixed-determinant}
\end{align}
The identity $\det(I-p\one^{\mathrm{T} })=1-\one^{\mathrm{T} }p$ follows directly by
multilinearity: terms selecting two columns proportional to $p$
vanish. Since $B-e_i\one^{\mathrm{T} }=DM$ and $\det D>0$ is constant,
\eqref{eq:determinant-sos} supplies the explicit formula
\begin{equation}
 g_i(Z)=
 \sum_{\substack{J\subseteq\{1,\ldots,s\}\\m+|J|\text{ even}}}
 \left[
 \sqrt{\det D}\,
 \Pf\begin{pmatrix}
  \Omega(Z)&\mathcal{R}_J(Z)^{\mathrm{T} }\\-\mathcal{R}_J(Z)&0
 \end{pmatrix}
 \right]^2.
\label{eq:fixed-pfaffian-sos}
\end{equation}
There is no polynomial denominator or additional multiplier.
There are $s=m(r+1)$ available rows in $\mathcal R$. Because the
lower-right block of each Pfaffian matrix is zero, only subsets
$J$ with $|J|\le m$ contribute. Together with the parity condition
$m+|J|$ even, this gives at most
$\sum_{0\le j\le m,\,j\equiv m\ (\mathrm{mod}\ 2)}\binom{s}{j}$
Pfaffian squares, possibly fewer after eliminating zero terms.

For the degree bound, each of the $|J|$ lower indices in a nonzero
paired product must pair with an upper index because the lower-right
block is zero. Thus $|J|\le m$. The product contains $|J|$ entries
from $\mathcal{R}_J$, of degree at most one, and
$(m-|J|)/2$ entries from $\Omega$, of degree at most two. Its total
degree is therefore at most $|J|+2(m-|J|)/2=m$.

Finally, \eqref{eq:fixed-determinant} and $h>0$ show that $g_i=0$
exactly when $\one^{\mathrm{T} }B^{-1}e_i=1$. Lemma~\ref{lem:deficit} with
$E=ZZ^{\mathrm{T} }$ makes this equivalent to $ZZ^{\mathrm{T} }=0$, hence to $Z=0$.
\end{proof}

\end{document}